\documentclass{amsart}

\usepackage[margin=1.35in]{geometry}
\usepackage{amsmath,amssymb}
\usepackage[colorlinks=true,linkcolor=blue,citecolor=blue,urlcolor=blue]{hyperref}
\usepackage{soul}

\newtheorem{theorem}{Theorem}[section]
\newtheorem{prop}[theorem]{Proposition}
\newtheorem{lem}[theorem]{Lemma}
\newtheorem{cor}[theorem]{Corollary}

\theoremstyle{definition}
\newtheorem{defn}[theorem]{Definition}
\newtheorem{ass}[theorem]{Assumption}
\newtheorem{cond}[theorem]{Condition}

\theoremstyle{remark}

\newtheorem{remark}[theorem]{Remark}

\numberwithin{equation}{section}

\newcommand{\be}{\begin{equation}}
\newcommand{\ee}{\end{equation}}

\newcommand{\lb}{\left(}
\newcommand{\rb}{\right)}
\newcommand{\lsb}{\left[}
\newcommand{\rsb}{\right]}
\newcommand{\lcb}{\left\{}
\newcommand{\rcb}{\right\}}

\newcommand{\ip}[1]{\langle#1\rangle}
\newcommand{\norm}[1]{\lVert#1\rVert}

\newcommand{\wt}{\widetilde}
\newcommand{\wh}{\widehat}

\newcommand{\ve}{\varepsilon}

\newcommand{\x}{f}
\newcommand{\y}{g}
\newcommand{\z}{h}

\newcommand{\B}{\mathcal{B}}

\newcommand{\E}{\mathbb{E}}
\newcommand{\F}{\mathcal{F}}
\newcommand{\J}{\mathbb{I}}
\renewcommand{\L}{{L}}

\newcommand{\N}{\mathbb{N}}
\renewcommand{\P}{\mathbb{P}}
\newcommand{\Q}{\mathbb{Q}}
\newcommand{\R}{\mathbb{R}}
\renewcommand{\S}{\mathcal{S}}
\newcommand{\U}{\mathcal{N}}
\newcommand{\V}{\mathcal{V}}
\newcommand{\W}{\mathcal{W}}
\newcommand{\X}{{X}}
\newcommand{\Y}{{Y}}
\newcommand{\Z}{{Z}}

\newcommand{\bm}{{W}}

\newcommand{\cadlag}{RCLL }

\newcommand{\push}{\ell}

\newcommand{\phib}{\mathcal{J}}
\newcommand{\etab}{\mathcal{K}}
\newcommand{\psib}{\mathcal{H}}

\newcommand{\lin}{\text{Lin}}

\newcommand{\allN}{\mathcal{I}}
\newcommand{\lip}{\kappa}

\newcommand{\spaan}{\text{span}}
\newcommand{\conv}{{\text{cone}}}

\newcommand{\dist}{\text{dist}}

\newcommand{\cts}{\mathbb{C}}
\newcommand{\dlr}{\mathbb{D}_{{\rm l,r}}}
\newcommand{\dr}{\mathbb{D}_{{\rm r}}}
\newcommand{\dlim}{\mathbb{D}_{{\rm lim}}}

\newcommand{\sphere}{\mathbb{S}^{J-1}}

\newcommand{\esm}{\bar\Gamma}
\newcommand{\dm}{\Lambda}

\newcommand{\proj}{\mathcal{L}}

\begin{document}

\title[Pathwise differentiability of reflected diffusions]{Pathwise differentiability of reflected diffusions in convex polyhedral domains}

\date{\today}

\author[David Lipshutz]{David Lipshutz*}
\address{Division of Applied Mathematics \\ 
        Brown University \\ 
        182 George Street, Providence \\ 
        RI 02912}
\email{david{\_}lipshutz@brown.edu}

\author[Kavita Ramanan]{Kavita Ramanan\textsuperscript{$\dagger$}}
\email{kavita{\_}ramanan@brown.edu}
\urladdr{http://www.dam.brown.edu/people/ramanan}



\thanks{*The first author was supported in part by NSF grant DMS-1148284, NSF grant CMMI-1407504 and AFOSR grant FA9550-12-1-0399.}
\thanks{\textsuperscript{$\dagger$}The second author was supported in part by NSF grant CMMI-1407504.}

\dedicatory{Brown University}

\begin{abstract}
Reflected diffusions in convex polyhedral domains arise in a variety of applications, including interacting particle systems, queueing networks, biochemical reaction networks and mathematical finance. Under suitable conditions on the data, we establish pathwise differentiability of such a reflected diffusion with respect to its defining parameters --- namely, its initial condition, drift and diffusion coefficients, and (oblique) directions of reflection along the boundary of the domain. We characterize the right-continuous regularization of a pathwise derivative of the reflected diffusion as the pathwise unique solution to a constrained linear stochastic differential equation with jumps whose drift and diffusion coefficients, domain and directions of reflection depend on the state of the reflected diffusion. The proof of this result relies on properties of directional derivatives of the associated (extended) Skorokhod reflection map and their characterization in terms of a so-called derivative problem, and also involves establishing certain path properties of the reflected diffusion at nonsmooth parts of the boundary of the polyhedral domain, which may be of independent interest. As a corollary, we obtain a probabilistic representation for derivatives of expectations of functionals of reflected diffusions, which is useful for sensitivity analysis of reflected diffusions.
\end{abstract}

\maketitle

\tableofcontents

\section{Introduction}\label{sec:intro}

\subsection{Overview}

Reflected diffusions in convex polyhedral domains arise in a variety of contexts, including in the study of interacting particle systems \cite{Burdzy2002,Warren2007,Zam2004}, rank-based diffusion models in mathematical finance \cite{Banner2005,Ichiba2011}, ``heavy traffic'' limits of stochastic networks \cite{Chen1991,Mandelbaum1998,Peterson1991,Ramanan2003,Ramanan2008a,Reiman1984}, and directed percolation and polymer models \cite{OConnell2001}. A reflected diffusion with given drift and diffusion coefficients in a convex polyhedral domain with oblique reflection along the boundary is a continuous Markov process that, roughly speaking, behaves like a diffusion (with the same drift and diffusion coefficients) in the interior of the domain and is constrained to remain in the closure of the domain by a ``constraining'' process that acts only when the reflected diffusion is at the boundary of the domain and only in specified directions that are constant along each boundary face of the polyhedral domain. Given the convex polyhedral domain, such a reflected diffusion is completely characterized by certain parameters --- namely, its initial condition, drift and diffusion coefficients, and directions of reflection (along the boundary).  The aim of this paper is to establish pathwise differentiability of the reflected diffusion with respect to all its defining parameters and provide a tractable characterization of the associated pathwise derivatives. Our work is motivated by both theoretical and applied perspectives. Pathwise differentiability with respect to the initial condition is closely related to differentiability of the transition semigroup and stochastic flows of reflected diffusions. Additionally,  in applications, pathwise derivatives are useful for characterizing sensitivities of expectations of functionals of reflected diffusions with respect to key model parameters, and this often entails simultaneous estimation of derivatives with respect to the drift and diffusion coefficients, as well as the directions of reflection (see \cite{Lipshutz2017a} for concrete examples). This motivates the development of a common framework within which to treat perturbations with respect to all defining parameters of a reflected diffusion. 

In the \emph{unconstrained} setting, the study of pathwise derivatives of a diffusion with respect to parameters that describe the diffusion is a classical topic in stochastic analysis (see, e.g., \cite{Kunita1997} and references therein), and has found a variety of applications, including in the estimation of so-called ``Greeks'' in math finance (see, e.g., \cite[Chapter 2]{Malliavin2006}). Under general conditions on the drift and diffusion coefficients, it is well known that the pathwise derivative of an \emph{unconstrained} diffusion is the unique solution to a linear stochastic differential equation (SDE) whose coefficients are modulated by the state of the diffusion (see, e.g., \cite{Metivier1982}). 

In the \emph{constrained} setting of reflected diffusions, pathwise derivatives are more complicated and no longer have continuous (or even right-continuous) paths. The analysis of pathwise differentiability is challenging due to the singular behavior of the constraining process or local time on the boundary of the domain, and additional challenges arise when the boundary is not smooth and reflection directions are multivalued at nonsmooth points of the boundary. In general, one would expect the derivative to satisfy a linearized version of the constrained SDE \eqref{eq:Z} for the reflected diffusion. However, it is a priori not even clear how to formally linearize the constrained SDE due to the presence of the constraining term and nonsmooth boundary. Furthermore, even when such an equation has been identified, additional challenges arise in establishing existence and uniqueness of solutions to this stochastic equation and showing that it indeed characterizes the pathwise derivative of the original equation. These tasks are further complicated in the presence of state-dependent drift and diffusion coefficients. We introduce a novel framework for analyzing pathwise derivatives of constrained processes, which in particular allows us to identify a suitable linearization. For a large class of reflected diffusions in convex polyhedral domains, our main result, Theorem \ref{thm:pathwise}, shows that (the right-continuous regularization of) a pathwise derivative can be characterized in terms of a so-called derivative process, which is shown to be the pathwise unique solution to a \emph{constrained} linear SDE with jumps whose drift and diffusion coefficients, domain and directions of reflection are modulated by the state of the reflected diffusion. Our characterization can thus be viewed as the analog, in the constrained setting, of the classical characterization of pathwise derivatives of diffusions. 

Our approach builds on the extended Skorokhod reflection problem (ESP) introduced in \cite{Ramanan2006} and the characterization of directional derivatives of the related extended Skorokhod map (ESM) obtained in \cite{Lipshutz2016}. The ESP provides an axiomatic framework with which to constrain a continuous deterministic path to the closure of a domain via a ``regulator'' function that acts in specified directions when the constrained path is at the boundary of the domain (see Definition \ref{def:esp} below). The mapping that takes the unconstrained path to the constrained path is referred to as the ESM. Under the ESP approach, the reflected diffusion is represented as the image of an unconstrained stochastic process under the ESM. For the class of domains and directions of reflection we consider, the corresponding ESM is Lipschitz continuous on path space, and (under standard Lipschitz continuity and nondegeneracy assumptions on the drift and diffusion coefficients) we show that pathwise derivatives of the reflected diffusion exist and can be characterized in terms of directional derivatives of the ESM (see Proposition \ref{prop:psivepsi} below). For a general class of ESMs in convex polyhedral domains, it was shown in \cite{Lipshutz2016} that directional derivatives exist along any continuous (deterministic) path whose constrained version satisfies a certain ``boundary jitter'' property (see Definition \ref{def:jitter} below), and their right-continuous regularizations can be uniquely characterized as solutions to a so-called derivative problem (see Definition \ref{def:dp} below).  For a smaller class of ESMs that satisfy a certain monotonicity property and admit a semi-explicit representation in terms of coupled one-dimensional Skorokhod maps, a characterization of directional derivatives along \emph{all} continuous paths was previously obtained in \cite{ManRam10}. However, this characterization lacks linearity properties satisfied by solutions to the derivative problem that yield a more tractable characterization of the pathwise derivative.

To establish pathwise differentiability of a reflected diffusion, we follow three main steps. First, in Section \ref{sec:jitterproof}, we prove that the reflected diffusions we consider almost surely satisfy the boundary jitter property, a result that may be of independent interest. The boundary jitter property (specifically conditions 3 and 4) describes the sample path behavior of a reflected diffusion immediately prior to hitting the nonsmooth parts of the boundary, and immediately after time zero if the reflected diffusion starts on the nonsmooth parts of the boundary. The proof of this property relies on uniform hitting time estimates, scaling properties of reflected diffusions, a change of measure and a weak convergence argument. When combined with results obtained in \cite{Lipshutz2016}, this implies that directional derivatives of the ESM evaluated at an associated unconstrained process exist almost surely. In the second step, carried out in Section \ref{sec:derivativeprocessproof}, we introduce the formal linearization of the original constrained SDE, which is expressed in terms of the so-called derivative problem introduced in \cite{Lipshutz2016}, and show that it admits a pathwise unique solution, which we refer to as the derivative process along the reflected diffusion. Finally, in Section \ref{sec:pathwiseproof}, we use stochastic estimates to show that pathwise derivatives can be expressed in terms of directional derivatives of the ESM, and characterize the right-continuous regularizations of the pathwise derivatives in terms of the corresponding derivative process along the reflected diffusion.

In summary, for a reflected diffusion in a large class of convex polyhedral domains, the main contributions of this work are as follows:
\begin{itemize}
	\item Verification of the boundary jitter property (Section \ref{sec:jitter} and Section \ref{sec:jitterproof}).
	\item Definition and analysis of the derivative process along the reflected diffusion (Section \ref{sec:derivativeprocess} and Section \ref{sec:derivativeprocessproof}).
	\item Existence of pathwise derivatives and their characterization via the derivative process (Section \ref{sec:pathwise} and Section \ref{sec:pathwiseproof}).
\end{itemize}
Our work appears to be the first to establish pathwise differentiability of reflected diffusions with state-dependent diffusion coefficients in nonsmooth domains, and also the first to consider perturbations of reflected diffusions with respect to diffusion coefficients and directions of reflection, which are of relevance in applications. For example, see \cite{Lipshutz2017a} where the results obtained here are used to construct new estimators for sensitivities of reflected diffusions.

\subsection{Prior results}

There are relatively few results on pathwise derivatives of \emph{obliquely} reflected diffusions in convex polyhedral domains. Two exceptions include the works of Andres \cite{Andres2009} and Dieker and Gao \cite{Dieker2014}. The work \cite{Andres2009} characterizes derivatives of flows of an obliquely reflected diffusion with identity diffusion coefficient in a polyhedral domain, but only up until the first time the reflected diffusion hits the nonsmooth part of the boundary. This avoids the difficulties that arise at the nonsmooth parts of the boundary, and essentially reduces the problem to studying differentiability of flows of obliquely reflected diffusions with identity diffusion coefficient in a half space. On the other hand, the work \cite{Dieker2014} considers the same class of ESMs in the nonnegative orthant studied in \cite{ManRam10}, and characterizes sensitivities of associated obliquely reflected diffusions to perturbations of the drift in the direction $-{\bf 1}$, the vector with negative one in each component.

In addition to these works, the following authors considered certain pathwise derivatives of \emph{normally} reflected diffusions: Deuschel and Zambotti \cite{Deuschel2005} characterized derivatives of stochastic flows for normally reflected diffusions with identity diffusion coefficient in the orthant; Burdzy \cite{Burdzy2009a} characterized derivatives of stochastic flows for normally reflected Brownian motions in \emph{smooth} domains; Andres \cite{Andres2011} generalized the results of \cite{Burdzy2009a} to allow state-dependent drifts; Pilipenko (see \cite{Pilipenko2013} and references therein) investigated derivatives of stochastic flows for normally reflected diffusions in the half space; and Bossy, Ciss\'e and Talay \cite{Bossy2011} obtained an explicit representation for the derivatives of a one-dimensional reflected diffusion in a bounded interval. Lastly, Costantini, Gobet and Karoui \cite{Costantini2006} studied boundary sensitivities of normally reflected diffusions in a time-dependent domain. 

\subsection{Outline of the paper}

In Section \ref{sec:coupleddiffusions}, we define a family of coupled reflected diffusions indexed by parameters that determine their initial conditions, drift and diffusion coefficients, and directions of reflection. In Section \ref{sec:main}, we present our main results. As explained prior to the summary of our main results above, Sections \ref{sec:jitterproof}, \ref{sec:derivativeprocessproof} and \ref{sec:pathwiseproof} are devoted to proving our main results. Appendices \ref{apdx:push}--\ref{apdx:ZkYkESPXk} contain the proofs of some auxiliary results. 

\subsection{Notation}\label{sec:notation}

We now collect some notation that will be used throughout this work. We let $\N=\{1,2,\dots\}$ denote the set of positive integers. Given $J\in\N$, we use $\R_+^J$ to denote the closed nonnegative orthant in $J$-dimensional Euclidean space $\R^J$. When $J=1$, we suppress $J$ and write $\R$ for $(-\infty,\infty)$ and $\R_+$ for $[0,\infty)$. We let $\Q$ denote the subset of rational numbers in $\R$. For a subset $A\subset\R$, we let $\inf A$ and $\sup A$ denote the infimum and supremum, respectively, of $A$. We use the convention that the infimum and supremum of the empty set are respectively defined to be $\infty$ and $-\infty$. For a column vector $x\in\R^J$, let $x^j$ denote the $j$th component of $x$. We write $\ip{\cdot,\cdot}$ and $|\cdot|$ for the usual Euclidean inner product and Euclidean norm, respectively, on $\R^J$. We let $\sphere\doteq\{x\in\R^J:|x|=1\}$ denote the unit sphere in $\R^J$ centered at the origin. For $J,K\in\N$, let $\R^{J\times K}$ denote the set of real-valued matrices with $J$ rows and $K$ columns. We write $M^T\in\R^{K\times J}$ for the transpose of a matrix $M\in\R^{J\times K}$. Given normed vector spaces $({\mathcal X},\norm{\cdot}_{\mathcal X})$ and $({\mathcal Y},\norm{\cdot}_{\mathcal Y})$, we let $\lin({\mathcal X},{\mathcal Y})$ denote the space of linear operators mapping ${\mathcal X}$ to ${\mathcal Y}$. For $T\in\lin({\mathcal X}, {\mathcal Y})$, we write the arguments of $T$ in square brackets to emphasize that $T$ is linear; that is, we write $T[x]$. For a bounded linear operator $T\in\lin({\mathcal X}, {\mathcal Y})$, we write $\norm{T}$ to denote the operator norm of $T$; that is, $\norm{T}\doteq\sup\{\norm{T[x]}_{\mathcal Y}:\norm{x}_{\mathcal X}=1\}$.

Given a subset $E\subseteq\R^J$, we let $\B(E)$ denote the Borel subsets of $E$. We let
	$$\conv(E)\doteq\lcb\sum_{k=1}^Kr_kx_k:K\in\N,x_k\in E,r_k\geq0\rcb,$$
denote the convex cone generated by $E$, and let $\spaan(E)$ denote the set of all possible finite linear combinations of vectors in $E$, with the convention that $\conv(\emptyset)$ and $\spaan(\emptyset)$ are equal to $\{0\}$. We let $E^\perp$ denote the orthogonal complement of $\spaan(E)$ in $\R^J$. We let $\dlim(E)$ denote the space of functions on $[0,\infty)$ taking values in $E$ that have finite left limits at every $t>0$ and finite right limits at every $t\geq0$. We let $\dlr(E)$ denote the subset of functions in $\dlim(E)$ that are either left-continuous or right-continuous at every $t>0$. We let $\dr(E)$ denote the subset of right-continuous functions in $\dlr(E)$ and refer to functions in $\dr(E)$ as right-continuous with finite left limits, or \cadlag for short. We let $\cts(E)$ denote the subset of continuous functions in $\dr(E)$. Given a subset $A\subseteq E$, we use $\cts_A(E)$ to denote the subset of continuous functions $f\in\cts(E)$ with $f(0)\in A$. We equip $\dlim(E)$ and its subsets with the topology of uniform convergence on compact intervals in $[0,\infty)$. For $f\in\dlr(E)$ and $t\in[0,\infty)$, define
	$$\norm{f}_t\doteq\sup_{s\in[0,t]}|f(s)|<\infty.$$
For $f\in\dlr(E)$, we let $f(t-)\doteq\lim_{s\uparrow t}f(s)$ for all $t>0$ and $f(t+)\doteq\lim_{s\downarrow t}f(s)$ for all $t\geq0$.

Throughout this paper we fix a filtered probability space $(\Omega,\F,\{\F_t\},\P)$ satisfying the usual conditions; that is, $(\Omega,\F,\P)$ is a complete probability space, $\F_0$ contains all $\P$-null sets in $\F$ and the filtration $\{\F_t\}$ is right-continuous. We write $\E$ to denote expectation under $\P$. By a $K$-dimensional $\{\F_t\}$-Brownian motion on $(\Omega,\F,\P)$, we mean that $\{\bm_t,\F_t, t \geq 0\}$ is a $K$-dimensional continuous martingale with quadratic variation $[\bm]_t=t$ for $t\geq0$ that starts at the origin. We let $C_p<\infty$, for $p\geq 2$, denote the universal constants in the Burkholder-Davis-Gundy (BDG) inequalities (see, e.g., \cite[Chapter IV, Theorem 42.1]{Rogers2000a}). 

We abbreviate ``almost surely'' as ``a.s.'' and ``infinitely often'' as ``i.o.''. 

\section{A parameterized family of reflected diffusions}\label{sec:coupleddiffusions}

In this section we introduce a family of coupled reflected diffusions in a convex polyhedral domain and describe their relation to the ESP.

\subsection{Description of the polyhedral domain}\label{sec:G}

Let $G$ be a nonempty convex polyhedron in $\R^J$ equal to the intersection of a finite number of closed half spaces in $\R^J$; that is,
	\be\label{eq:G}G\doteq\bigcap_{i=1,\dots,N}\lcb x\in\R^J:\ip{x,n_i}\geq c_i\rcb,\ee
for some positive integer $N\in\N$, unit vectors $n_i\in \sphere$ and constants $c_i\in\R$, for $i=1,\dots,N$. We assume the representation for $G$ given in \eqref{eq:G} is minimal in the sense that the intersection of half spaces $\{x\in\R^J:\ip{x,n_i}\geq c_i\}$ over $i$ in any strict subset of $\{1,\dots,N\}$ is not equal to $G$. For each $i=1,\dots,N$, we let $F_i\doteq\{x\in\partial G:\ip{x,n_i}=c_i\}$ denote the $i$th face. For notational convenience, we let $\allN\doteq\{1,\dots,N\}$, and for $x\in G$, we write $\label{eq:allNx}\allN(x)\doteq\{i\in\allN:x\in F_i\}$ to denote the (possibly empty) set of indices associated with the faces that intersect at $x$. Given a subset $I\subseteq\allN$, we let $|I|$ denote the cardinality of the set $I$. For each $x\in\partial G$, we let
	$$n(x)\doteq\conv(\{n_i,i\in\allN(x)\})$$
denote the cone of inward normals to the polyhedron $G$ at $x$. For notational convenience, we extend the definition of $n(x)$ to all of $G$ by setting $n(x)\doteq\{0\}$ for $x\in G^\circ$.

\subsection{Introduction of parameters and definition of a reflected diffusion}\label{sec:reflecteddiffusion}

Let $M\in\N$ and let the parameter set $U$ be an open subset of $\R^M$. For each $i\in\allN$, fix a continuously differentiable mapping
	$$d_i:U\mapsto\R^J$$ 
satisfying $\ip{d_i(\alpha),n_i}>0$ for all $\alpha\in U$. For a given parameter $\alpha\in U$, $d_i(\alpha)$ denotes the associated direction of reflection along the face $F_i$. Since the directions of reflection can always be renormalized (while also preserving the continuous differentiability in $\alpha$ of the normalized mapping), we assume without loss of generality that $\ip{d_i(\alpha),n_i}=1$ for all $\alpha\in U$. For $\alpha\in U$ and $x\in\partial G$, we let $d(\alpha,x)$ denote the cone generated by the admissible directions of reflection at $x$; that is,
	\be\label{eq:dx}d(\alpha,x)\doteq\conv\lb\lcb d_i(\alpha),i\in\allN(x)\rcb\rb.\ee
For convenience, we extend the definition of $d(\alpha,\cdot)$ to all of $G$ by setting $d(\alpha,x)\doteq\{0\}$ for $x\in G^\circ$. Fix continuously differentiable functions 
\begin{align*}
	b:U\times G\mapsto\R^J,\qquad\sigma:U\times G\mapsto\R^{J\times K},
\end{align*}
and denote their respective Jacobians by $b':U\times G\mapsto\lin(\R^M\times\R^J,\R^J)$ and $\sigma':U\times G\mapsto\lin(\R^M\times\R^J,\R^{J\times K})$. For each $\alpha\in U$, $b(\alpha,\cdot)$ and $a(\alpha,\cdot)\doteq\sigma(\alpha,\cdot)\sigma^T(\alpha,\cdot)$, respectively, denote the drift and diffusion coefficients of the reflected diffusion associated with the parameter $\alpha$.

\begin{defn}\label{def:rd}
Given $\alpha\in U$, $\{(d_i(\alpha),n_i,c_i),i\in\allN\}$, $b(\alpha,\cdot)$, $\sigma(\alpha,\cdot)$, $x\in G$ and a $K$-dimensional $\{\F_t\}$-Brownian motion on $(\Omega,\F,\P)$, a reflected diffusion associated with the parameter $\alpha$, initial condition $x$ and driving Brownian motion $W$ is a $J$-dimensional continuous $\{\F_t\}$-adapted process $\Z^{\alpha,x}=\{\Z_t^{\alpha,x},t\geq0\}$ such that a.s.\ for all $t\geq0$, $\Z_t^{\alpha,x}\in G$ and
	\be\label{eq:Z}\Z_t^{\alpha,x}=x+\int_0^t b(\alpha,\Z_s^{\alpha,x})ds+\int_0^t\sigma(\alpha,\Z_s^{\alpha,x})d\bm_s+\Y_t^{\alpha,x},\ee
where $\Y^{\alpha,x}=\{\Y_t^{\alpha,x},t\geq0\}$ is a $J$-dimensional continuous $\{\F_t\}$-adapted process that a.s.\ satisfies $\Y_0^{\alpha,x}=0$ and, for all $0\leq s<t<\infty$,
	\be\label{eq:Y}\Y_t^{\alpha,x}-\Y_s^{\alpha,x}\in\conv\lsb\cup_{u\in(s,t]}d(\alpha,\Z_u^{\alpha,x})\rsb.\ee
We refer to $\Y^{\alpha,x}$ as the constraining process associated with $\Z^{\alpha,x}$. 
\end{defn}

Conditions under which a pathwise unique reflected diffusion exists are specified in Proposition \ref{prop:rdeu} below. 

\begin{remark}\label{rmk:Valpha}
In \cite[Theorem 4.3]{Ramanan2006} it was shown that a.s.\ $\Y^{\alpha,x}$ has finite total variation on compact subsets of the stochastic interval $[0,\tau_0^{\alpha,x})$, where $\tau_0^{\alpha,x}$ is the first hitting time of the set
	\be\label{eq:Valpha}\V^\alpha\doteq\partial G\setminus\{x\in\partial G:\;\exists\; n\in n(x)\text{ such that }\ip{n,d}>0,\;\forall\;d\in d(\alpha,x)\setminus\{0\}\}.\ee
(The set $\V^\alpha$ has a different definition in \cite[equation (2.15)]{Ramanan2006}; however, an examination of the proof of \cite[Theorem 4.3]{Ramanan2006} reveals that the result holds with $\V^\alpha$ defined as in \eqref{eq:Valpha}.) Consequently, when $\V^\alpha$ is empty, a.s.\ the total variation of $\Y^{\alpha,x}$ on $[0,t]$, denoted $|\Y^{\alpha,x}|(t)$, is finite for all $t<\infty$ and, as shown in \cite[Lemma 2.7]{Kang2017}, there exists a measurable function $\xi^{\alpha,x}:(\Omega\times[0,\infty),\F\otimes\B([0,\infty)))\mapsto(\R^J,\B(\R^J))$ such that a.s.\ for all $0\leq s<t<\infty$,
	$$\Y_t^{\alpha,x}-\Y_s^{\alpha,x}=\int_{[s,t]}\xi_u^{\alpha,x}d|\Y^{\alpha,x}|(u),$$
and $\xi_u^{\alpha,x}\in d(\alpha,\Z_u^{\alpha,x})$ for $d|\Y^{\alpha,x}|$-almost every $u\geq0$.
\end{remark}

\subsection{The extended Skorokhod reflection problem}\label{sec:esp}

In this section we state the ESP (for continuous paths) and recall conditions under which the associated ESM is well defined. The ESP was introduced in \cite{Ramanan2006} as a pathwise method for constructing reflected diffusions.  It is a generalization of the Skorokhod problem that allows for a constraining term that potentially has unbounded variation on compact intervals. Even when the constraining term is of bounded variation on compact intervals, the ESP formulation is often more convenient since the associated ESM has desirable properties such as a closed graph. In particular, the ESP formulation more naturally leads to the identification of a suitable  linearized version that characterizes pathwise derivatives of constrained processes (see the similarity between the definition of the ESP and that of the derivative problem given in Definition \ref{def:dp}).

\begin{defn}\label{def:esp}
Let $\alpha\in U$. Given $\x\in\cts_G(\R^J)$, $(\z,\y)\in\cts(G)\times\cts(\R^J)$ solves the ESP $\{(d_i(\alpha),n_i,c_i),i\in\allN\}$ for $\x$ if $\z(0)=\x(0)$, and if for all $t\geq0$ the following properties hold:
\begin{itemize}
	\item[1.] $\z(t)=\x(t)+\y(t)$;
	\item[2.] for every $s\in[0,t)$, 
		$$\y(t)-\y(s)\in\conv\lsb\cup_{u\in(s,t]}d(\alpha,\z(u))\rsb.$$
\end{itemize}
If there exists a unique solution $(\z,\y)$ to the ESP $\{(d_i(\alpha),n_i,c_i),i\in\allN\}$ for $\x$, we write $\z=\esm^\alpha(\x)$ and refer to $\esm^\alpha$ as the ESM associated with the ESP $\{(d_i(\alpha),n_i,c_i),i\in\allN\}$.
\end{defn}

\begin{remark}\label{rmk:X}
Given $\alpha\in U$, $x\in G$ and a reflected diffusion $\Z^{\alpha,x}$ as in Definition \ref{def:rd}, define the $J$-dimensional continuous $\{\F_t\}$-adapted process $\X^{\alpha,x}=\{\X_t^{\alpha,x},t\geq0\}$ by
	\be\label{eq:X}\X_t^{\alpha,x}\doteq x+\int_0^tb(\alpha,\Z_s^{\alpha,x})ds+\int_0^t\sigma(\alpha,\Z_s^{\alpha,x})d\bm_s,\qquad t\geq0.\ee
Then by the properties stated in Definition \ref{def:rd} and the statement of the ESP in Definition \ref{def:esp}, a.s.\ $(\Z^{\alpha,x},\Y^{\alpha,x})$ is a solution to the ESP $\{(d_i(\alpha),n_i,c_i),i\in\allN\}$ for $\X^{\alpha,x}$.
\end{remark}

We now provide geometric conditions on the data $\{(d_i(\alpha),n_i,c_i),i\in\allN\}$ under which the ESM is well defined on $\cts_G(\R^J)$. The first condition, Assumption \ref{ass:setB} below, was introduced in \cite[Assumption 2.1]{Dupuis1991} and ensures the ESM is Lipschitz continuous on its domain of definition.

\begin{ass}\label{ass:setB}
For each $\alpha\in U$ there exists $\delta_\alpha>0$ and a compact, convex, symmetric set $B^\alpha$ in $\R^J$ with $0\in(B^\alpha)^\circ$ such that for $i\in\allN$,
	\be\label{eq:setB}\lcb\begin{array}{l}z\in\partial B^\alpha\\|\ip{z,n_i}|<\delta_\alpha\end{array}\rcb\qquad\Rightarrow\qquad\ip{\nu,d_i(\alpha)}=0\qquad\text{for all }\;\nu\in\nu_{B^\alpha}(z),\ee
where $\nu_{B^\alpha}(z)$ denotes the set of inward normals to the set $B^\alpha$ at $z$; that is,
	$$\nu_{B^\alpha}(z)\doteq\{\nu\in\sphere:\ip{\nu,y-z}\geq0\text{ for all }y\in B^\alpha\}.$$
\end{ass}

\begin{prop}\label{prop:esmlip}
Suppose Assumption \ref{ass:setB} holds. For each $\alpha\in U$ there exists $\lip_{\esm}(\alpha)<\infty$ such that if $(\z_1,\y_1)$ is a solution to the ESP for $\x_1\in\cts_G(\R^J)$ and $(\z_2,\y_2)$ is a solution to the ESP for $\x_2\in\cts_G(\R^J)$, then for all $t<\infty$,
\begin{align*}
	\norm{\z_1-\z_2}_t+\norm{\y_1-\y_2}_t&\leq\lip_{\esm}(\alpha)\norm{\x_1-\x_2}_t.
\end{align*}
\end{prop}

\begin{proof}
This follows from \cite[Theorem 3.3]{Ramanan2006} and property 1 of Definition \ref{def:esp}.
\end{proof}

The next condition requires that for each $\alpha\in U$ there is a projection mapping from $\R^J$ to $G$ satisfying certain geometric conditions related to the associated directions of reflection.

\begin{ass}\label{ass:projection}
For each $\alpha\in U$ there is a function $\pi^\alpha:\R^J\mapsto G$ satisfying $\pi^\alpha(x)=x$ for all $x\in G$ and $\pi^\alpha(x)-x\in d(\alpha,\pi^\alpha(x))$ for all $x\not\in G$.
\end{ass}

See \cite[Section 2.3]{Lipshutz2016} and references therein for examples of broad classes of ESPs that satisfy Assumptions \ref{ass:setB} and \ref{ass:projection}.

\begin{prop}\label{prop:esp}
Suppose Assumptions \ref{ass:setB} and \ref{ass:projection} hold. Then for each $\alpha\in U$ and $\x\in\cts_G(\R^J)$, there exists a unique solution $(\z,\y)$ to the ESP $\{(d_i(\alpha),n_i,c_i),i\in\allN\}$ for $\x$. In particular, for each $\alpha\in U$ the ESM $\esm^\alpha$ is well defined on $\cts_G(\R^J)$.
\end{prop}

\begin{proof}
The proposition follows from Proposition \ref{prop:esmlip} and \cite[Lemma 2.6]{Ramanan2006}.
\end{proof}

\begin{remark}\label{rmk:picontinuous}
Under Assumptions \ref{ass:setB} and \ref{ass:projection}, for each $\alpha\in U$, the function $\pi^\alpha:\R^J\mapsto G$ is continuous (see the discussion in \cite[Section 5.3]{Dupuis1991}).
\end{remark}

\subsection{The reflection matrix}

Define $R:U\mapsto\R^{J\times N}$ by
	$$\label{eq:R}R(\alpha)\doteq\begin{pmatrix}d_1(\alpha)&\cdots&d_N(\alpha)\end{pmatrix},$$
and let $R':U\mapsto\lin(\R^M,\R^{J\times N})$ denote the Jacobian of $R$, which is well defined since $d_i(\cdot)$ are continuously differentiable. For each $\alpha\in U$, we refer to $R(\alpha)$ as the \emph{reflection matrix} associated with $\alpha$. Under the following linear independence assumption on the directions of reflection, given a solution $(\z,\y)$ to the ESP, there is a unique decomposition of a constraining path $\y$ into an $N$-dimensional path that describes its action along each face.

\begin{cond}\label{cond:independent}
For $\alpha\in U$ and $x\in\partial G$, $\{d_i(\alpha),i\in\allN(x)\}$ is a set of linearly independent vectors.
\end{cond}

\begin{lem}\label{lem:push}
Suppose Condition \ref{cond:independent} holds. Let $\alpha\in U$. Given a solution $(\z,\y)$ to the ESP $\{(d_i(\alpha),n_i,c_i),i\in\allN\}$ for $\x\in\cts_G(\R^J)$, there exists a unique function $\push\in\cts(\R_+^N)$ such that $\y=R\push$ and for each $i\in\allN$, $\push^i(0)=0$, $\push^i$ is nondecreasing and $\push^i$ can only increase when $\z$ lies in face $F_i$; that is,
	\be\label{eq:dLi}\int_0^\infty1_{\{\z(s)\not\in F_i\}}d\push^i(s)=0.\ee
Consequently, $\y$ has finite variation on compact time intervals.  Moreover, there exists $\lip_\ell(\alpha)<\infty$ such that if, for $k=1,2$, $(\z_k,\y_k)$ is the solution to the ESP $\{(d_i(\alpha),n_i,c_i),i\in\allN\}$ for $\x_k\in\cts_G$ and $\push_k$ is as above, but with $\z_k,\y_k$ and $\push_k$ in place of $\z,\y$ and $\push$, then for all $t<\infty$, 
	\be\label{eq:pushL}\norm{\push_1-\push_2}_t\leq\lip_\ell(\alpha)\norm{\y_1-\y_2}_t, \ee
and, in addition, $\lip_\ell(\cdot)$ can be chosen so that it is bounded on compact subsets of $U$. Furthermore, given $\wt\alpha\in U$, $(\z,\y)$ is solution to the ESP $\{(d_i(\wt\alpha),n_i,c_i),i\in\allN\}$ for $\wt\x\doteq\x+(R(\alpha)-R(\wt\alpha))\ell$. 
\end{lem}

When $N=J$ and there exists $x\in G$ such that $\cap_{i=1,\dots,J}F_i=\{x\}$ (e.g., if $G=\R_+^J$), Condition \ref{cond:independent} implies that $R(\alpha)$ is invertible and upon setting $\push\doteq(R(\alpha))^{-1}\y$ it is readily verified that the first statement of the lemma follows from condition 2 of Definition \ref{def:esp}. The proof of Lemma \ref{lem:push} in full generality is given in Appendix \ref{apdx:push}. 

\begin{remark}\label{rmk:local}
Suppose Condition \ref{cond:independent} holds. Then according to Remark \ref{rmk:X} and Lemma \ref{lem:push}, given a reflected diffusion $\Z^{\alpha,x}$ with associated constraining process $\Y^{\alpha,x}$, there is a unique $N$-dimensional continuous $\{\F_t\}$-adapted process $\L^{\alpha,x}=\{\L_t^{\alpha,x},t\geq0\}$ such that a.s.\ $\Y^{\alpha,x}=R(\alpha)\L^{\alpha,x}$, $\L_0^{\alpha,x}=0$ and for each $i\in\allN$, the $i$th component $[\L^{\alpha,x}]^i$ is nondecreasing and can only increase when $\Z^{\alpha,x}$ lies in face $F_i$. In particular, $\Y^{\alpha,x}$ has finite variation on compact time intervals, so by \eqref{eq:Z}, $\Z^{\alpha,x}$ is a semimartingale.
\end{remark}

We impose the following assumption to ensure that when the directions of reflection are nonconstant in $\alpha\in U$, then there is a unique decomposition of the constraining process into its action along each face, in the sense of Lemma \ref{lem:push}.

\begin{ass}\label{ass:DRL}
At least one of the following holds:
\begin{itemize}
	\item Condition \ref{cond:independent} and there exists $\lip_R<\infty$ such that $\norm{R'(\alpha)}\leq\lip_R$ for all $\alpha\in U$.
	\item $R(\alpha)$ is constant in $\alpha\in U$.
\end{itemize}
\end{ass}

\begin{remark}\label{rmk:DRL}
Suppose Assumption \ref{ass:DRL} holds. Given $\alpha\in U$, $x\in G$, a reflection diffusion $\Z^{\alpha,x}$ with associated constraining process $\Y^{\alpha,x}$, and $\beta\in\R^M$, for conciseness in the statements of proofs and theorems, with some abuse of notation we interpret the $J$-dimensional continuous process $\{R'(\alpha)[\beta]\L_t^{\alpha,x},t\geq0\}$ as follows:
\begin{itemize}
	\item If Condition \ref{cond:independent} holds, then $\L^{\alpha,x}$ denotes the $N$-dimensional continuous process described in Remark \ref{rmk:local}.
	\item On the other hand, if Condition \ref{cond:independent} does not hold, then $R(\alpha)$ is constant in $\alpha\in U$ and we interpret $\{R'(\alpha)[\beta]\L_t^{\alpha,x},t\geq0\}$ to be identically zero (even though the process $\L_t^{\alpha,x}$ may not be well defined).
\end{itemize}
\end{remark}

\subsection{Existence and uniqueness of reflected diffusions}\label{sec:rbmeu}

In this section we recall a well known result that guarantees strong existence and pathwise uniqueness of reflected diffusions.

\begin{ass}\label{ass:drift}
There exists $\lip_{b,\sigma}<\infty$ such that $\norm{b'(\alpha,x)}+\norm{\sigma'(\alpha,x)}\leq\lip_{b,\sigma}$ for all $\alpha\in U$ and $x\in G$.
\end{ass}

\begin{prop}[{\cite[Theorem 4.3]{Ramanan2006}}]
\label{prop:rdeu}
Given $\{(d_i(\cdot),n_i,c_i),i\in\allN\}$, $b(\cdot,\cdot)$ and $\sigma(\cdot,\cdot)$, suppose Assumptions \ref{ass:setB}, \ref{ass:projection} and \ref{ass:drift} hold. Then for each $\alpha\in U$, $x\in G$ and $K$-dimensional $\{\F_t\}$-Brownian motion on $(\Omega,\F,\P)$, there exists a reflected diffusion $\Z^{\alpha,x}$ associated with the parameter $\alpha$, initial condition $x$ and driving Brownian motion $\bm$, and $\Z^{\alpha,x}$ is a strong Markov process. Furthermore, if $\wt\Z^{\alpha,x}$ is another reflected diffusion associated with the parameter $\alpha$, initial condition $x$ and driving Brownian motion $\bm$, then a.s.\ $\Z^{\alpha,x}=\wt\Z^{\alpha,x}$. In other words, pathwise uniqueness holds.
\end{prop}

We close this section by showing there exists a modification of the family of reflected diffusions that is continuous in its various parameters. Given $\Z^{\alpha,x}$ recall the definition of $\X^{\alpha,x}$ in \eqref{eq:X}.

\begin{lem}\label{lem:ZalphaxZbetaybounds}
Given $\{(d_i(\cdot),n_i,c_i),i\in\allN\}$, $b(\cdot,\cdot)$ and $\sigma(\cdot,\cdot)$, suppose Assumptions \ref{ass:setB}, \ref{ass:projection}, \ref{ass:DRL} and \ref{ass:drift} hold. Then for each $p\geq 2$, $t<\infty$ and compact subsets $V\subset U$ and $K\subset G$, there exist constants $C^\dagger,C^\ddagger<\infty$, depending only on $p$, $t$, $V$ and $K$, such that for all $(\alpha,x),(\wt\alpha,\wt x)\in V\times K$,
\begin{align}\label{eq:ZpKolmogorov}
	\E\lsb\norm{\Z^{\alpha,x}-\Z^{\wt\alpha,\wt x}}_t^p\rsb&\leq C^\dagger|\alpha-\wt\alpha|^p+C^\ddagger|x-\wt x|^p,\\ \label{eq:XpKolmogorov}
	\E\lsb\norm{\X^{\alpha,x}-\X^{\wt\alpha,\wt x}}_t^p\rsb&\leq\wt C^\dagger|\alpha-\wt\alpha|^p+\wt C^\ddagger|x-\wt x|^p.
\end{align}
\end{lem}

The proof of Lemma \ref{lem:ZalphaxZbetaybounds} is given in Appendix \ref{sec:continuousfield}.

\begin{prop}\label{prop:continuousfield}
Given $\{(d_i(\cdot),n_i,c_i),i\in\allN\}$, $b(\cdot,\cdot)$ and $\sigma(\cdot,\cdot)$, suppose Assumptions \ref{ass:setB}, \ref{ass:projection}, \ref{ass:DRL} and \ref{ass:drift} hold. Then there is a modification of the random field $\{\Z_t^{\alpha,x},\alpha\in U,x\in G,t\geq0\}$ such that for each $\omega\in\Omega$, $(\alpha,x,t)\mapsto\Z_t^{\alpha,x}(\omega)$ is continuous as a function from $U\times G\times[0,\infty)$ to $G$.
\end{prop}

\begin{proof}
Since for each $\alpha\in U$ and $x\in G$, $\Z^{\alpha,x}$ takes values in $\cts(\R^J)$, it suffices to show there is a modification of $\Z^{\alpha,x}$ such that for each $\omega\in\Omega$, $(\alpha,x)\mapsto\Z^{\alpha,x}(\omega)$ is continuous as a mapping from $U\times G$ to $\cts(\R^J)$. This follows from Kolmogorov's continuity criterion on random fields (see, e.g., \cite[Theorem 1.4.1]{Kunita1997}) and Lemma \ref{lem:ZalphaxZbetaybounds}.
\end{proof}

\section{Main results}\label{sec:main}

In Section \ref{sec:jitter} we introduce the boundary jitter property and show that a reflected diffusion satisfies this property under a uniform ellipticity condition on the diffusion coefficient. In Section \ref{sec:derivativeprocess} we introduce the derivative process along a reflected diffusion and establish pathwise uniqueness. In Section \ref{sec:pathwise} we present our main result on the existence of pathwise derivatives of a reflected diffusion and their characterization via a derivative process.

\emph{Throughout the remainder of this work we assume the coefficients $b(\cdot,\cdot)$ and $\sigma(\cdot,\cdot)$ satisfy Assumption \ref{ass:drift} and the data $\{(d_i(\cdot),n_i,c_i),i\in\allN\}$ satisfies Assumptions \ref{ass:setB}, \ref{ass:projection} and \ref{ass:DRL}, and only state additional assumptions made, where required. We fix a $K$-dimensional $\{\F_t\}$-Brownian motion on $(\Omega,\F,\P)$ and for each $\alpha\in U$ and $x\in G$, we let $\Z^{\alpha,x}$ denote the pathwise unique reflected diffusion associated with the parameter $\alpha$, initial condition $x$ and driving Brownian motion $\bm$. According to Proposition \ref{prop:continuousfield}, there is a continuous modification of the field $\{\Z_t^{\alpha,x},\alpha\in U,x\in G,t\geq0\}$. We work with this continuous modification. Let $\Y^{\alpha,x}$ denote the associated constraining process introduced in Definition \ref{def:rd}, Let $\X^{\alpha,x}$ denote the process defined in \eqref{eq:X}, $\Y^{\alpha,x}\doteq\X^{\alpha,x}-\Z^{\alpha,x}$ so that $(\Z^{\alpha,x},\Y^{\alpha,x})$ is the solution to the ESP $\{(d_i(\alpha),n_i,c_i),i\in\allN\}$ for $\X^{\alpha,x}$, and, for $\beta\in\R^M$, let $\{R'(\alpha)[\beta]\L_t^{\alpha,x},t\geq0\}$ denote the process described in Remark \ref{rmk:DRL}.}

\subsection{Boundary jitter property}\label{sec:jitter}

In this section we state the boundary jitter property, which was first introduced in \cite{Lipshutz2016}. Let
	\be\label{eq:U}\U\doteq\{x\in\partial G:|\allN(x)|\geq2\}\ee
denote the nonsmooth part of the boundary $\partial G$, and
	\be\label{eq:S}\S\doteq\partial G\setminus\U=\{x\in\partial G:|\allN(x)|=1\}\ee
denote the smooth part of the boundary $\partial G$.

\begin{defn}\label{def:jitter}
We say that $(\z,\y)\in\cts(G)\times\cts(\R^J)$ satisfies the boundary jitter property if the following conditions hold:
\begin{itemize}
	\item[1.] If $t\geq0$ is such that $\z(t)\in\S$, then for all $t_1<t<t_2$, $\y$ is nonconstant on $(t_1\vee0,t_2)$.
	\item[2.] The path $\z$ does not spend positive Lebesgue time in $\U$; that is,
		$$\int_0^\infty 1_{\U}(\z(t))dt=0.$$
	\item[3.] If $\z(t)\in\U$ for some $t>0$, then for each $i\in\allN(\z(t))$ and every $\delta\in(0,t)$, there exists $s\in(t-\delta,t)$ such that $\allN(\z(s))=\{i\}$.
	\item[4.] If $\z(0)\in\U$, then for each $i\in\allN(\z(0))$ and every $\delta>0$, there exists $s\in(0,\delta)$ such that $\allN(\z(s))=\{i\}$.
\end{itemize}
\end{defn}

Condition 3 of the jitter property states that if a path hits a point in the nonsmooth part of the boundary at some time $t>0$, then it must hit the relative interior of all the adjoining faces infinitely often in any interval just prior to $t$, whereas condition 4 is a time-reversed version of condition 3 that states that a path 
starting at a point in $\U$ must hit all adjoining faces infinitely often in any interval just after time $t=0$.

Under the following uniform ellipticity condition on the diffusion coefficient, we show that a reflected diffusion along with its constraining process satisfies the boundary jitter property.

\begin{ass}\label{ass:elliptic}
For each $\alpha\in U$ there exists $\lambda(\alpha)>0$ such that for all $x\in G$,
	$$y^Ta(\alpha,x)y\geq\lambda(\alpha)|y|^2,\qquad y\in\R^J.$$
\end{ass}

\begin{theorem}\label{thm:jitter}
Under Assumption \ref{ass:elliptic}, for each $\alpha\in U$ and $x\in G$ a.s.\ $(\Z^{\alpha,x},\Y^{\alpha,x})$ satisfies the boundary jitter property.
\end{theorem}

\begin{remark}
The proof of Theorem \ref{thm:jitter} does not require that Assumption \ref{ass:DRL} hold, but does require that  the other standing assumptions stated at the beginning of Section \ref{sec:main} hold. 
\end{remark}

The proof of Theorem \ref{thm:jitter} is given in Section \ref{sec:jitterproof}. 

\subsection{Derivative process}\label{sec:derivativeprocess}

We now introduce a derivative process along a reflected diffusion. To define its domain, we set, for each $x\in\partial G$,
	\be\label{eq:Hx}H_x\doteq\bigcap_{i\in\allN(x)}\lcb y\in\R^J:\ip{y,n_i}=0\rcb.\ee
For $x\in G^\circ$, set $H_x\doteq\R^J$. Given $x\in\partial G$, perturbations of $x$ in directions that lie in $H_x$ remain in the same subset of faces that $x$ lies in; that is, if $y\in H_x$, then $\allN(x+\ve y)=\allN(x)$ for all $\ve>0$ sufficiently small. As shown in Theorem \ref{thm:pathwise}, it suffices to consider only such perturbations. For the following, given $\alpha\in U$, $x\in G$ and $\beta\in\R^M$, recall the interpretation of the $J$-dimensional process $\{R'(\alpha)[\beta]\L_t^{\alpha,x},t\geq0\}$ given in Remark \ref{rmk:DRL}.

\begin{defn}\label{def:de}
Let $\alpha\in U$ and $x\in G$. A derivative process along $\Z^{\alpha,x}$ is an $\{\F_t\}$-adapted \cadlag process $\phib^{\alpha,x}=\{\phib_t^{\alpha,x},t\geq0\}$ taking values in $\lin(\R^M\times H_x,\R^J)$ that a.s.\ satisfies for all $t\geq0$ and $(\beta,y)\in\R^M\times H_x$, $\phib_t^{\alpha,x}[\beta,y]\in H_{\Z_t^{\alpha,x}}$ and
\begin{align}\label{eq:de}
	\phib_t^{\alpha,x}[\beta,y]&=y+\int_0^tb'(\alpha,\Z_s^{\alpha,x})[\beta,\phib_s^{\alpha,x}[\beta,y]]ds+\int_0^t\sigma'(\alpha,\Z_s^{\alpha,x})[\beta,\phib_s^{\alpha,x}[\beta,y]]d\bm_s\\ \notag
	&\qquad+R'(\alpha)[\beta]\L_t^{\alpha,x}+\etab_t^{\alpha,x}[\beta,y],
\end{align}
where $\etab^{\alpha,x}=\{\etab_t^{\alpha,x},t\geq0\}$ is an $\{\F_t\}$-adapted \cadlag process taking values in $\lin(\R^M\times H_x,\R^J)$ such that a.s.\ for all $(\beta,y)\in\R^M\times H_x$, $\etab_0^{\alpha,x}[\beta,y]=0$ and for all $0\leq s<t<\infty$,
	\be\label{eq:etab}\etab_t^{\alpha,x}[\beta,y]-\etab_s^{\alpha,x}[\beta,y]\in\spaan\lsb\cup_{u\in(s,t]}d(\alpha,\Z_u^{\alpha,x})\rsb.\ee
\end{defn}

As mentioned in the introduction, a derivative process satisfies a constrained linear SDE with jumps of the form \eqref{eq:de} whose drift and diffusion coefficients, domain and directions of reflection all depend on the state of the reflected diffusion. To understand its dynamics, note that on time intervals when $Z^{\alpha,x}$ lies in the interior of the domain, the last two terms in \eqref{eq:de} are constant and hence, $\phib^{\alpha,x}[\beta,y]$ evolves (continuously) according to a linear SDE, whose coefficients are modulated by the process $Z^{\alpha,x}$.  At any time $t>0$ when $Z^{\alpha,x}_t$ hits the boundary $\partial G$, the conditions $\phib_t^{\alpha,x}[\beta,y]\in H_{\Z_t^{\alpha,x}}$ and \eqref{eq:etab} ensure that $\phib^{\alpha,x}_t[\beta,y]$ is the image of $\phib_{t-}^{\alpha,x}[\beta,y]$ under a certain linear ``derivative projection'' operator ${\mathcal L}_{\Z_t^{\alpha,x}}^\alpha$ that depends only on the faces (and the associated directions of reflection) on which  $\Z_t^{\alpha,x}$ lies (see Lemma \ref{lem:projx} below). 

We close this section with conditions ensuring pathwise uniqueness of a derivative process. Existence of derivative processes will follow from Theorem \ref{thm:pathwise}.

\begin{theorem}\label{thm:dpunique}
Let $\alpha\in U$ and $x\in G$. Suppose $\phib^{\alpha,x}$ and $\wt\phib^{\alpha,x}$ are derivative processes along $\Z^{\alpha,x}$. Then a.s.\ $\phib^{\alpha,x}=\wt\phib^{\alpha,x}$. In other words, pathwise uniqueness holds.
\end{theorem}

The proof of Theorem \ref{thm:dpunique} is deferred to Section \ref{sec:derivativeprocessproof}. 

\subsection{Pathwise derivatives}\label{sec:pathwise}

The main result of this section is Theorem \ref{thm:pathwise}, which characterizes pathwise derivatives of reflected diffusions. First, we introduce an additional regularity assumption. Recall that we use $\norm{\cdot}$ to denote the operator norm.

\begin{ass}\label{ass:Holder}
There exists $\lip'<\infty$ and $\gamma\in(0,1]$ such that for all $\alpha,\beta\in U$ and $x,y\in G$,
	\be\label{eq:Holder}\norm{b'(\alpha,x)-b'(\beta,y)}+\norm{\sigma'(\alpha,x)-\sigma'(\beta,y)}+\norm{R'(\alpha)-R'(\beta)}\leq\lip'|(\alpha,x)-(\beta,y)|^\gamma.\ee
\end{ass}

Given $x\in\partial G$, define 
	\be\label{eq:Gx}G_x\doteq\bigcap_{i\in\allN(x)}\{y\in\R^J:\ip{y,n_i}\geq0\},\ee
and for $x\in G^\circ$, set $G_x\doteq\R^J$. Then $G_x$ describes the local structure of the polyhedron $G$ at $x$ and denotes the directions in which we allow the initial condition $x$ to be perturbed. In particular, since $U$ is an open set and due to \eqref{eq:G} and \eqref{eq:Gx}, given $\alpha\in U$, $x\in G$, $\beta\in\R^M$ and $y\in G_x$, there exists $\ve_0(\alpha,x,\beta,y)>0$ sufficiently small such that
	 \be\label{eq:ve0}\alpha+\ve\beta\in U,\qquad x+\ve y\in G,\qquad\text{for all }0<\ve<\ve_0(\alpha,x,\beta,y).\ee
For such $\ve>0$ sufficiently small, define the continuous process $\partial_{\beta,y}^\ve\Z^{\alpha,x}=\{\partial_{\beta,y}^\ve\Z_t^{\alpha,x},t\geq0\}$ by
	\be\label{eq:nablaveZ}\partial_{\beta,y}^\ve\Z^{\alpha,x}\doteq\frac{\Z^{\alpha+\ve\beta,x+\ve y}-\Z^{\alpha,x}}{\ve}.\ee
In Theorem \ref{thm:pathwise} below, we characterize a.s.\ limits of \eqref{eq:nablaveZ} as $\ve\downarrow0$. First, we have the following bound on the moments of $\partial_{\beta,y}^\ve\Z^{\alpha,x}$.

\begin{lem}\label{lem:UI}
Given $\alpha\in U$, $x\in G$, $\beta\in\R^M$ and $y\in G_x$, let $\ve_0\doteq\ve_0(\alpha,x,\beta,y)>0$ be sufficiently small such that \eqref{eq:ve0} holds. Then for each $p\geq2$ and $t<\infty$,
	\be\label{eq:nablaZmomentbound}\sup\lcb\E\lsb\norm{\partial_{\beta,y}^\ve\Z^{\alpha,x}}_t^p\rsb:0<\ve<\ve_0\rcb<\infty.\ee
Consequently,
	\be\label{eq:uniformintegrability}\lim_{C\to\infty}\sup\lcb\E\lsb\norm{\partial_{\beta,y}^\ve\Z^{\alpha,x}}_t1_{\lcb\norm{\partial_{\beta,y}^\ve\Z^{\alpha,x}}_t\geq C\rcb}\rsb:0<\ve<\ve_0\rcb=0.\ee
\end{lem}

\begin{proof}
Define the compact sets $V\doteq\{\alpha+\ve\beta,0\leq\ve\leq\ve_0\}$ and $K\doteq\{x+\ve y,0\leq \ve\leq\ve_0\}$. The moment bound \eqref{eq:nablaZmomentbound} then follows from Lemma \ref{lem:ZalphaxZbetaybounds} with $\wt\alpha=\alpha+\ve\beta$ and $\wt x=x+\ve y$, and the uniform integrability shown in \eqref{eq:uniformintegrability} is an immediate consequence of \eqref{eq:nablaZmomentbound}.
\end{proof}

In order to establish existence of pathwise derivatives of reflected diffusions, we require that the reflected diffusion not hit a certain subset of the boundary of the domain. For $\alpha\in U$, define
	\be\label{eq:Walpha}\W^\alpha\doteq\{x\in\U:\spaan(d(\alpha,x)\cup H_x)\neq\R^J\},\ee
and
	\be\label{eq:tau}\tau^{\alpha,x}\doteq\inf\{t\geq0:\Z_t^{\alpha,x}\in\W^\alpha\}.\ee 
In \cite{Lipshutz2016} it was shown that directional derivatives of the ESM $\esm^\alpha$ may not exist at times that the constrained path lies in $\W^\alpha$ (see \cite[Appendix D.2]{Lipshutz2016} for an example) and so we require that a.s.\ $\tau^{\alpha,x}=\infty$. This is not too stringent a requirement since, as the next lemma shows, under a mild linear independence condition on the directions of reflection, the set $\W^\alpha$ is empty.

\begin{lem}[{\cite[Lemma 8.2]{Lipshutz2016}}]
\label{lem:Wempty}
Suppose Condition \ref{cond:independent} holds. Then $\W^\alpha$ is empty for all $\alpha\in U$.
\end{lem}

\begin{remark}
There are cases under which Condition \ref{cond:independent} does not hold and nevertheless, $\W^\alpha$ is empty (see, e.g., \cite[Section D.1]{Lipshutz2016}). Even when $\W^\alpha$ is not empty, there are cases where a.s.\ $\tau^{\alpha,x}=\infty$. For instance, consider a reflected Brownian motion in a two-dimensional wedge with vertex at the origin and equal directions of reflection along both faces of the wedge (this corresponds to the setting in \cite{Varadhan1984} with $\theta_1+\theta_2+\pi=\xi$). In this case $\W^\alpha=\{0\}$. However, according to \cite[Theorem 2.2]{Varadhan1984}, if the reflected Brownian motion does not start at the origin, then a.s.\ $\tau^{\alpha,x}=\infty$.
\end{remark}

In the next lemma, we recall the definition of a so-called derivative projection operator introduced in \cite{Lipshutz2016} to characterize directional derivatives of the ESM.

\begin{lem}[{\cite[Lemma 8.3]{Lipshutz2016}}]
\label{lem:projx}
Given $\alpha\in U$ and $x\in G\setminus\W^\alpha$, there exists a unique mapping $\proj_x^\alpha :\R^J\mapsto H_x$ that satisfies $\proj_x^\alpha(y)-y\in\spaan[d(\alpha,x)]$ for all $y\in\R^J$. Furthermore, $\proj_x^\alpha$ is a linear map.
\end{lem}

\begin{remark}\label{rmk:projx}
Let $\alpha\in U$. By the uniqueness of the mapping $\proj_x^\alpha$, we have $\proj_x^\alpha[y]=y$ for all $y\in H_x$. When $x\in G^\circ$, it follows that $H_x=\R^J$ and $\proj_x^\alpha$ reduces to the identity operator on $\R^J$.
\end{remark}

We can now state our main result.

\begin{theorem}\label{thm:pathwise}
Suppose Assumption \ref{ass:Holder} holds. Let $\alpha\in U$ and $x\in G\setminus\W^\alpha$. Suppose that a.s.\ $\tau^{\alpha,x}=\infty$ and $(\Z^{\alpha,x},\Y^{\alpha,x})$ satisfies the boundary jitter property. Then there exists a pathwise unique derivative process $\phib^{\alpha,x}$ along $\Z^{\alpha,x}$ and for all $\beta\in\R^M$ and $y\in G_x$, a.s.\ the following hold:
\begin{itemize}
	\item[(i)] The pathwise derivative of $\Z^{\alpha,x}$ in the direction $(\beta,y)$, defined for $t\geq0$ by 
		\be\label{eq:nablaZ}\partial_{\beta,y}\Z_t^{\alpha,x}\doteq\lim_{\ve\downarrow0}\frac{\Z_t^{\alpha+\ve\beta,x+\ve y}-\Z_t^{\alpha,x}}{\ve},\ee
	exists.
	\item[(ii)] The pathwise derivative $\partial_{\beta,y}\Z^{\alpha,x}=\{\partial_{\beta,y}\Z_t^{\alpha,x},t\geq0\}$ takes values in $\dlr(\R^J)$ and is continuous at times $t>0$ when $\Z_t^{\alpha,x}\in G^\circ\cup\U$. 
	\item[(iii)] The right-continuous regularization of the pathwise derivative $\partial_{\beta,y}\Z^{\alpha,x}$ is equal to the derivative process $\phib^{\alpha,x}$ evaluated in the direction $(\beta,\proj_x^\alpha[y])$; that is,
		\be\label{eq:pathwisederivativeprocess}\lim_{s\downarrow t}\partial_{\beta,y}\Z_s^{\alpha,x}=\phib_t^{\alpha,x}[\beta,\proj_x^\alpha[y]],\qquad t\geq0.\ee
\end{itemize}
\end{theorem}

\begin{remark}
The derivative projection operator $\proj_x^\alpha$ in part (iii) of Theorem \ref{thm:pathwise} serves to map $y\in G_x$, the direction in which $x$ is perturbed, to $H_x$ (the domain of $\phib^{\alpha,x}[\beta,\cdot]$). The presence of $\proj_x^\alpha$ in part (iii) can be interpreted as stating that any perturbation to the initial condition $x$ in a direction that lies in $G_x\setminus H_x$ is instantly projected to the linear subspace $H_x$ along a direction that lies in $\spaan[d(\alpha,x)]$.
\end{remark}

The proof of Theorem \ref{thm:pathwise} is given in Section \ref{sec:pathwiseproof}. When combined with Theorem \ref{thm:jitter} and Lemma \ref{lem:Wempty}, we have the following corollary.

\begin{cor}\label{cor:pathwise}
Suppose Condition \ref{cond:independent} and Assumptions \ref{ass:elliptic} and \ref{ass:Holder} hold. Let $\alpha\in U$ and $x\in G$. Then there exists a pathwise unique derivative process and for all $\beta\in\R^M$ and $y\in G_x$, a.s.\ (i), (ii) and (iii) of Theorem \ref{thm:pathwise} hold.
\end{cor}

Let $\zeta_1:G\mapsto\R$ and $\zeta_2:G\mapsto\R$ be continuously differentiable functions with bounded first partial derivatives and denote their respective Jacobians by $\zeta_1':G\mapsto\lin(\R^J,\R)$ and $\zeta_2':G\mapsto\lin(\R^J,\R)$. Let $t>0$ and define $\Theta:U\times G\mapsto\R$ by
	\be\label{eq:Falphax} \Theta(\alpha,x)\doteq\E\lsb\int_0^t\zeta_1(\Z_s^{\alpha,x})ds+\zeta_2(\Z_t^{\alpha,x})\rsb,\qquad(\alpha,x)\in U\times G\ee
Such quantities arise in applications and it is of interest in sensitivity analysis to compute the Jacobian of $\Theta(\alpha,x)$. The following corollary provides a stochastic representation for the Jacobian of $\Theta(\alpha,x)$.

\begin{cor}\label{cor:sensitivity}
Let $\alpha\in U$ and $x\in G^\circ$. Suppose Condition \ref{cond:independent} and Assumptions \ref{ass:elliptic} and \ref{ass:Holder} hold. Then $\Theta$ is differentiable at $(\alpha,x)$ and its Jacobian at $(\alpha,x)$, denoted $\Theta'(\alpha,x)$, satisfies
	\be\label{eq:DF}\Theta'(\alpha,x)=\E\lsb\int_0^t\zeta_1'(\Z_s^{\alpha,x})[\phib_s^{\alpha,x}]ds+\zeta_2'(\Z_t^{\alpha,x})[\phib_t^{\alpha,x}]\rsb.\ee
\end{cor}

\begin{proof}
Fix $(\alpha,x)\in U\times G^\circ$. Given $\beta\in\R^M$ and $y\in\R^J$, by \eqref{eq:Falphax}, the Lipschitz continuity of $\zeta_1$ and $\zeta_2$, the uniform integrability shown in Lemma \ref{lem:UI} and part (i) of Theorem \ref{thm:pathwise}, we have
\begin{align*}
	\lim_{\ve\to0}\frac{\Theta(\alpha+\ve\beta,x+\ve y)-\Theta(\alpha,x)}{\ve}&=\E\lsb\int_0^t\zeta_1'(\Z_s^{\alpha,x})[\nabla_{\beta,y}\Z_s^{\alpha,x}]ds+\zeta_2'(\Z_t^{\alpha,x})[\nabla_{\beta,y}\Z_t^{\alpha,x}]\rsb.
\end{align*}
Then, due to parts (ii) and (iii) of Theorem \ref{thm:pathwise} and the facts that $\P(\Z_t^{\alpha,x}\in G^\circ)=1$ (see Lemma \ref{lem:ZtGcirc} below), $\partial_{\beta,y}\Z^{\alpha,x}$ is continuous at $t$ if $\Z_t^{\alpha,x}\in G^\circ$ and $\proj_x^\alpha$ is equal to the identity operator when $x\in G^\circ$ (see Remark \ref{rmk:projx}), we have
\begin{align*}
	\lim_{\ve\to0}\frac{\Theta(\alpha+\ve\beta,x+\ve y)-\Theta(\alpha,x)}{\ve}&=\E\lsb\int_0^t\zeta_1'(\Z_s^{\alpha,x})[\phib_s^{\alpha,x}[\beta,y]]ds+\zeta_2'(\Z_t^{\alpha,x})[\phib_t^{\alpha,x}[\beta,y]]\rsb.
\end{align*}
Since $\beta\in\R^M$ and $y\in\R^J$ were arbitrary and the right-hand side of the last display is linear in $(\beta,y)$, this concludes the proof of the corollary.
\end{proof}

The representation \eqref{eq:DF} suggests pathwise methods for estimating $\Theta'(\alpha,x)$, which we develop in subsequent work \cite{Lipshutz2017a}. Pathwise estimators (also referred to as infinitesimal perturbation analysis estimators) are usually preferable when available (see, e.g., the discussion at the end of \cite[Chapter 7]{Glasserman2003}). For instance, they have smaller bias than finite difference estimators. In addition, likelihood ratio estimators, which rely on a change of measure argument (see, e.g., \cite{Yang1991}), only apply to perturbations of the drift because perturbations to the initial condition, diffusion coefficient or directions of reflection typically do not preserve absolute continuity of the law of the perturbed process with respect to the law of the unperturbed process.  

\section{Verification of the boundary jitter property}\label{sec:jitterproof}

In this section we prove Theorem \ref{thm:jitter}, which provides conditions under which a reflected diffusion, along with its constraining process, satisfies the boundary jitter property stated in Definition \ref{def:jitter}. As explained in Section \ref{sec:jitter12}, the first two conditions of the boundary jitter property can be deduced  in a fairly straightforward manner from existing results. On the other hand, the verifications of the last two conditions of the
boundary jitter property, which are carried out in Section \ref{sec:jitter34}, are considerably more complicated. For a class of reflected Brownian motions in the quadrant, these conditions were established in \cite{Lipshutz2016}. A (non-trivial) generalization of that argument could also be used to establish the jitter property for reflected Brownian motions in higher dimensional polyhedral domains under suitable conditions. However, the study of the boundary jitter property in the presence of state-dependent diffusion coefficients requires a completely new approach. Our proof of this property in this more general setting relies on some uniform hitting time estimates, Lipschitz continuity of ESMs associated with certain reduced versions of the original ESP, and weak convergence arguments that are established in Section \ref{sec:weakconvergence}.

Throughout this section we fix $\alpha\in U$. For convenience, we suppress the ``$\alpha$'' dependence and write $\Z^x$, $\Y^x$ and $\X^x$ in place of $\Z^{\alpha,x}$, $\Y^{\alpha,x}$ and $\X^{\alpha,x}$, respectively.

\subsection{Verifications of conditions 1 and 2}\label{sec:jitter12}

We first prove condition 1 of the boundary jitter property in the case $b(\alpha,\cdot)\equiv0$. In the proof of Theorem \ref{thm:jitter} below, we use a change of measure argument to show that condition 1 holds for general Lipschitz continuous drifts.

\begin{lem}\label{lem:jitter1}
Suppose Assumption \ref{ass:elliptic} holds and $b(\alpha,\cdot)\equiv0$. Let $x\in G$. Almost surely, if $t\geq0$ is such that $\Z_t^x\in\partial G$, then for all $t_1,t_2\in\R$ satisfying $t_1<t<t_2$, $\Y^x$ is nonconstant on $(t_1\vee0,t_2)$. In other words, a.s.\ $(\Z^x,\Y^x)$ satisfies condition 1 of the boundary jitter property.
\end{lem}

\begin{proof}
Consider the events
	\be\label{eq:A0x}A_0^x\doteq\bigcap_{t_2\in\Q\cap(0,\infty)}\lcb \Y^x\text{ is nonconstant on }(0,t_2)\rcb,\ee
and
	$$A^x\doteq\bigcap_{t_1\in\Q\cap(0,\infty)}\bigcap_{t_2\in\Q\cap(t_1,\infty)}\lcb \Y^x\text{ is nonconstant on }(t_1,t_2)\rcb\cup\lcb\Z_t^x\in G^\circ\text{ for }t_1<t<t_2\rcb.$$
Then we need to show that $\P(A_0^x)=1$ for all $x\in\partial G$ and $\P(A^x)=1$ for all $x\in G$.	
	
Suppose $x\in\partial G$. Let $i\in\allN(x)$ and $t_2\in\Q\cap(0,\infty)$. By \eqref{eq:Z}, \eqref{eq:X} and because $\Y_0^x=0$ and $\Z_t^x\in G$ for all $t\geq0$,
	\be\label{eq:Yxconstant}\{\Y^x\text{ is constant on }(0,t_2)\}=\{\ip{\X_t^x-x,n_i}=\ip{\Z_t^x-x,n_i}\geq 0\text{ for }0< t< t_2\}.\ee
By \eqref{eq:X}, the fact that $b(\alpha,\cdot)\equiv0$ and Assumption \ref{ass:elliptic}, $\{\ip{\X_t-x,n_i},t\geq0\}$ is a one-dimensional continuous local martingale starting at zero with quadratic variation
	\be\label{eq:Xnqvar}[\ip{\X^x-x,n_i}]_t\doteq\int_0^tn_i^Ta(\alpha,\Z_s^x)n_ids\geq\lambda(\alpha) t,\qquad t\geq0.\ee
Therefore, by \cite[Chapter IV, Theorem 34.1]{Rogers2000a}, there is a (one-dimensional) Brownian motion $B=\{B_t,t\geq0\}$ such that $\ip{\X_t^x-x,n_i}=B_{[\ip{\X^x,n_i}]_t}$ for $t\geq0$. Thus, by \eqref{eq:Yxconstant} and \eqref{eq:Xnqvar},
	\be\label{eq:PYxconstant}\P(\Y^x\text{ is constant on }(0,t_2))\leq\P\lb B_t\geq0\text{ for }0<t<\lambda(\alpha)t_2\rb=0,\ee
where the final equality is a well known property of Brownian motion. Together, \eqref{eq:A0x} and \eqref{eq:PYxconstant} imply $\P(A_0^x)=1$.

Now suppose $x\in G$. Fix $t_1\in\Q\cap(0,\infty)$ and $t_2\in\Q\cap(t_1,\infty)$. Define the $\{\F_t\}$-stopping time $\rho\doteq\inf\{t>t_1:\Z_t^x\in\partial G\}$. Note that $\{\rho\geq t_2\}=\{\Z_s^x\in G^\circ\text{ for }t_1<t<t_2\}$, so we are left to consider the event $\{t_1\leq\rho<t_2\}$. By the strong Markov property, $\{\Y_{\rho+t}^x-\Y_{\rho}^x,t\geq0\}$ conditioned on $\{\rho\in[t_1,t_2),\Z_\rho^x=y\}$ is equal in distribution to $\Y^y$. Since, as shown above, $\P(A_0^y)=1$ for all $y\in\partial G$, we have, for $t_3\in\Q\cap(0,\infty)$,
\begin{align*}
	&\P(\Y_{\rho+\cdot}^x\text{ is nonconstant on }(0,t_3)|t_1\leq\rho<t_2)\\
	&\qquad=\int_{\partial G}\P(\Y^y\text{ is nonconstant on }(0,t_3))\P(\Z_\rho^x\in dy|t_1\leq\rho<t_2)=1.
\end{align*}
Since the above holds for every $t_3\in\Q\cap(0,\infty)$, $\P(A^x)=1$, which completes the proof.
\end{proof}

\begin{lem}\label{lem:jitter2}
Suppose Assumption \ref{ass:elliptic} holds. Let $x\in G$. Then
 	\be\label{eq:jitter2}\P\lb\int_0^\infty1_{\partial G}(\Z_s^x)ds=0\rb=1.\ee
In particular, a.s.\ $\Z^x$ satisfies condition 2 of the boundary jitter property.
\end{lem}

\begin{proof}
Due to the definition of $G$ given in \eqref{eq:G} as the intersection of finitely many half spaces, it is clear that $\V^\alpha$ defined in \eqref{eq:Valpha} is the union of finitely many closed disconnected sets. Then by Proposition \ref{prop:rdeu} and \cite[Theorem 2]{Kang2017}, the law of the process $\Z^x$ induced on $\cts(G)$ (equipped with its $\sigma$-algebra of Borel subsets) is a solution to the associated submartingale problem starting at $x$ (see \cite[Definition 2.9]{Kang2017}). The lemma then follows from \cite[Proposition 2.12]{Kang2017}.
\end{proof}

\subsection{Uniform hitting time estimates}\label{sec:weakconvergence}

Before verifying conditions 3 and 4 of the boundary jitter property, we establish estimates on certain hitting times. We first consider the case that the drift satisfies $b(\alpha,\cdot)\equiv0$, which is assumed throughout this section. To handle the case of general drift coefficients, we can then use a change of measure argument (see the proof of Theorem \ref{thm:jitter} below).

Since conditions 3 and 4 of the boundary jitter property hold automatically when $\U$, the nonsmooth part of the boundary $\partial G$,  is empty, we assume $\U$ is nonempty. Set $\J\doteq\{\allN(y):y\in\U\}$. For $I\in\J$, define the nonempty set
	\be\label{eq:FI} F_I\doteq\bigcap_{i\in I}F_i,\ee
where recall that $F_i = \{x \in \partial G: \langle x, n_i \rangle  = c_i\}$ for $i\in\allN$. Let $\dist(\cdot,\cdot)$ denote the usual Euclidean metric
on $\R^J$, and given $x \in\R^J, A \subset \R^J$, let $\dist(x,A) \doteq \inf_{y \in A} \dist (x,y)$. Define the $\{\F_t\}$-stopping times 
\begin{align}\label{eq:thetaix}
	\theta_i^x&\doteq\inf\{t>0:\Z_t^x\in F_i\},\\ \label{eq:sigmaIx}
	\sigma_I^x&\doteq\inf\{t>0:\dist(\Z_t^x,F_I)\leq\dist(x,F_I)/2\},\\ \label{eq:tauIx}
	\tau_I^x&\doteq\inf\{t>0:\dist(\Z_t^x,F_I)\geq2\dist(x,F_I)\}.
\end{align} 
For $r>0$ define 
	\be\label{eq:SIr}\mathbb{S}_I(r)\doteq\{y\in G:\dist(y,F_I)=r\}.\ee
Define the decreasing sequence $\{r_k\}_{k\in\N}$ in $(0,\infty)$ by 
	\be\label{eq:rk}r_k\doteq 2^{-k},\qquad k\in\N.\ee 
Since $G$ is convex with nonempty interior, $\mathbb{S}_I(r)$ is nonempty for all $r>0$ sufficiently small. Without loss of generality we assume $\mathbb{S}_I(r_k)$ is nonempty for all $k\in\N$. 

The following is the main hitting time estimate of this section.

\begin{prop}\label{prop:weakconvergence}
Let $I\in\J$ and $\bar x\in F_I$. Suppose $\{x_k\}_{k\in\N}$ is a sequence in $G$ satisfying 
	\be\label{eq:xkSrk}x_k\in\mathbb{S}_I(r_k),\; k\in\N,\qquad\text{and}\qquad\lim_{k\to\infty}x_k=\bar{x}.\ee
Then for each $i\in I$,
\begin{align}\label{eq:sigmaIxksigmaI}
	\liminf_{k\to\infty}\P\lb\theta_i^{x_k}<\sigma_I^{x_k}\rb&>0,\\ \label{eq:tauIxktauI}
	\liminf_{k\to\infty}\P\lb\theta_i^{x_k}<\tau_I^{x_k}\rb&>0.
\end{align}
\end{prop}

The remainder of this section is devoted to the proof of Proposition \ref{prop:weakconvergence}. Fix $I\in\J$, $\bar x\in F_I$ and a sequence $\{x_k\}_{k\in\N}$ in $G$ satisfying \eqref{eq:xkSrk}. Let $i\in I$. For $x\in G$, define the $\{\F_t\}$-stopping time
	\be\label{eq:rhoIx}\rho_I^x\doteq\inf\lcb t>0:\Z_t^x\in\cup_{j\in\allN\setminus I}F_j\rcb\ee
to be the first time $\Z^x$ lies in a face that does not contain $\bar{x}$. For each $k\in\N$, consider the scaled processes defined by 
\begin{align}\label{eq:bmk}
	\bm_t^k&\doteq\frac{\bm_{r_k^2t}}{r_k},\\ \label{eq:Zk}
	\Z_t^k&\doteq\frac{\Z_{r_k^2t\wedge\rho_I^{x_k}}^{x_k}-\bar{x}}{r_k},\\ \label{eq:Xk}
	\X_t^k&\doteq\frac{\X_{r_k^2t\wedge\rho_I^{x_k}}^{x_k}-\bar{x}}{r_k},\\ \label{eq:Yk}
	\Y_t^k&\doteq\frac{\Y_{r_k^2t\wedge\rho_I^{x_k}}^{x_k}}{r_k},
\end{align}
for $t\geq0$. 

\begin{remark}\label{rmk:Zkesm}
For each $k\in\N$, it follows from Definition \ref{def:esp}, \eqref{eq:Zk}--\eqref{eq:Yk} and a straightforward verification argument that a.s.\ $(\Z^k,\Y^k)$ is a solution to the ESP $\{(d_i(\alpha),n_i,0),i\in I\}$ for $\X^k$. 
\end{remark}

Let $\F_t^k\doteq\F_{r_k^2t}$ for $t\geq0$. Clearly, the processes $\bm^k$, $\Z^k$, $\X^k$ and $\Y^k$ are $\{\F_t^k\}$-adapted. Define the $\{\F_t^k\}$-stopping times
\begin{align}\label{eq:thetaik1}
	\theta_i^k&\doteq\inf\{t>0:\ip{\Z_t^k,n_i}=0\},\\ \label{eq:sigmaIk1}
	\sigma_I^k&\doteq\inf\{t>0:\dist(\Z_t^k,F_I)\leq1/2\},\\ \label{eq:tauIk1}
	\tau_I^k&\doteq\inf\{t>0:\dist(\Z_t^k,F_I)\geq2\}.
\end{align}
Due to \eqref{eq:thetaix}, \eqref{eq:sigmaIx}, \eqref{eq:tauIx}, \eqref{eq:Zk}, the fact that $\bar x\in F_I$, and the fact that $\dist(x_k,F_I)=r_k$ by \eqref{eq:xkSrk} and \eqref{eq:SIr}, we have
\begin{align}\label{eq:thetaixrk2thetaik}
	\theta_i^{{x_k}}&=r_k^2\theta_i^k,\qquad\sigma_I^{{x_k}}=r_k^2\sigma_I^k,\qquad\tau_I^{{x_k}}=r_k^2\tau_I^k.
\end{align}
In the following remark, we sketch the proof of Proposition \ref{prop:weakconvergence} in a simple case where the argument is relatively straight forward. 

\begin{remark}
Suppose the diffusion coefficient is constant (i.e., $\sigma(\alpha,\cdot)\equiv\sigma(\alpha)$ so $\Z^x$ is a reflected Brownian motion), $G$ is a convex cone with vertex at the origin (i.e., $F_{\allN}=\{0\}$), $I=\allN$ (so $\bar x=0$) and there exists $x_0\in\mathbb{S}_I(1)$ such that $x_k=r_kx_0$ for $k\in\N$. For $k\in\N$, by \eqref{eq:Zk}, the facts that $\Z^{x_k}$ takes values in the convex cone $G$, \eqref{eq:bmk} and \eqref{eq:Yk}, it is readily verified that $(\Z^k,\Y^k)$ is a solution to the ESP $\{(d_i(\alpha),n_i,0),i\in\allN)\}$ for $\X^k=x_0+\sigma(\alpha)\bm^k$. It follows from \eqref{eq:bmk} and Brownian scaling that $\X^k\buildrel{d}\over=\X^1$ for all $k\in\N$. This, along with the measurability of $\esm^\alpha$ and the fact that $\Z^k=\esm(\X^k)$ for all $k\in\N$, implies $\Z^k\buildrel{d}\over=\Z^1$ for all $k\in\N$. Thus, by \eqref{eq:thetaik1}--\eqref{eq:thetaixrk2thetaik}, we have
\begin{align*}
	\P\lb\theta_i^{x_k}<\sigma_I^{x_k}\rb&=\P\lb\theta_i^k<\sigma_I^k\rb=\P\lb\theta_i^1<\sigma_I^1\rb,\\
	\P\lb\theta_i^{x_k}<\tau_I^{x_k}\rb&=\P\lb\theta_i^k<\tau_I^i\rb=\P\lb\theta_i^1<\tau_I^1\rb.
\end{align*}
The estimates \eqref{eq:sigmaIxksigmaI} and \eqref{eq:tauIxktauI} then follow once we show the probabilities on the right-hand side of the above display are positive, which  follows from the nondegeneracy of the diffusion coefficient stated in Assumption \ref{ass:elliptic}. Since the argument is similar to the one carried out in the proof of Proposition \ref{prop:weakconvergence} below, we omit the details here. The proof for the case when $\bar{x}$ lies on another face $F_I$, $I \subsetneq {\mathcal I}$, is more complicated, even when the diffusion coefficient is constant.
\end{remark}

The proof of Proposition \ref{prop:weakconvergence} in the general setting of state-dependence covariance is considerably more involved. We first state the following helpful lemmas. 

\begin{lem}[{\cite[Lemma 2.1]{Kang2007}}]
\label{lem:allNusc}
For each $x\in G$, there is an open neighborhood $V_x$ of $x$ in $\R^J$ such that
	\be\label{eq:allNusc}\allN(y)\subseteq\allN(x)\qquad\text{for all }y\in V_x\cap G.\ee
\end{lem}

Recall the definition of $\rho_I^x$ in \eqref{eq:rhoIx}. 

\begin{lem}\label{lem:rholsc}
Almost surely, the mapping $x\mapsto\rho_I^x$ from $G$ to $\R_+$ is lower semicontinuous.
\end{lem}
\begin{proof}
Fix $\omega\in\Omega$. Let $x\in G$ and $\{x_\ell\}_{\ell \in\N}$ be a sequence in $G$ such that $x_\ell\to x$ as $\ell \to\infty$. If $\rho_I^x(\omega)=0$, then $\liminf_{\ell \to\infty}\rho_I^{x_\ell}(\omega)\geq\rho_I^x(\omega)$. On the other hand, if $\rho_I^x(\omega)>0$, let $t<\rho_I^x(\omega)$. By \eqref{eq:rhoIx}, the fact that $\cup_{j\in\allN\setminus I}F_j$ is a closed set and the continuity of $x\mapsto Z^x(\omega)$ (Proposition \ref{prop:continuousfield}), $\liminf_{\ell\to\infty}\rho_I^{x_\ell}(\omega)>t$. Since this holds for all $t<\rho_I^x(\omega)$, we have $\liminf_{\ell\to\infty}\rho_I^{x_\ell}(\omega)\geq\rho_I^x(\omega)$, which completes the proof.
\end{proof}

\begin{lem}\label{lem:rhok}
Almost surely, $\frac{1}{r_k^2}\rho_I^{x_k}\to\infty$ as $k\to\infty$.
\end{lem}

\begin{proof}
Due to \eqref{eq:xkSrk}, \eqref{eq:rk} and Lemma \ref{lem:allNusc}, there exists $k_0\in\N$ such that $x_k\in\mathbb{B}\doteq\{y\in G:|y-\bar x|\leq r_{k_0}\}$ for all $k\geq k_0$, and $\allN(x)\subseteq\allN(\bar x)=I$ for all $x\in\mathbb{B}$. Thus, for each $x\in\mathbb{B}$, \eqref{eq:rhoIx} and the continuity of $\Z^x$ imply that a.s.\ 
	\be\label{eq:chix}\rho_I^x>\chi^x\doteq\inf\{t>0:\Z^x\not\in\mathbb{B}\}\geq0.\ee 
Then due to Lemma \ref{lem:rholsc} and the compactness of $\mathbb{B}$, a.s.\ $\inf\{\rho_I^x:x\in\mathbb{B}\}>0$. Therefore, given $\ve>0$, there exists $\delta>0$ such that $\P\lb\inf\{\rho_I^x:x\in\mathbb{B}\}>\delta\rb\geq1-\ve$. For each $k\geq k_0$, using the fact that $x_k\in\mathbb{B}$, \eqref{eq:chix} and the strong Markov property for $\Z^x$, we have,
	$$\P\lb\rho_I^{x_k}>\delta\rb\geq\P\lb\inf\{\rho_I^x:x\in\mathbb{B}\}>\delta\rb\geq1-\ve.$$
Let $C<\infty$ and choose $k_1\geq k_0$ such that $r_{k_1}^2\leq\delta/C$. Then for all $k\geq k_1$,
	$$\P\lb\frac{1}{r_k^2}\rho_I^{x_k}\leq C\;\forall\;k\in\N\rb\leq\P\lb\rho_I^{x_{k_1}} \leq r_{k_1}^2C\rb\leq\P\lb \inf\{ \rho_I^{x}: x \in \mathbb{B}\} \leq \delta \rb\leq\ve.$$
Since $\ve>0$ and $C<\infty$ were arbitrary, the conclusion of the lemma follows.
\end{proof}

In the following remark we observe that there exists a simple equivalence between $(\Z^k,\X^k,\Y^k)$ and another triplet of processes that will be easier to work with.

\begin{remark}
For $k\in\N$, using Brownian scaling, we can define a Brownian motion $\wh\bm^k=\{\wh\bm_t^k,t\geq0\}$ by 
	\be\label{eq:wtbmk}\wh\bm_t^k\doteq r_k\bm_{t/r_k^2},\qquad t\geq0.\ee
Let $\wh\F_t^k\doteq\F_{t/r_k^2}$ for $t\geq0$ so that $\wh\bm^k$ is $\{\wh\F_t^k\}$-adapted. Let $\wh\Z^{k,x_k}$ denote the reflected diffusion in $G$ with initial condition $x_k$, coefficients $b(\alpha,\cdot)\equiv0$ and $\sigma(\alpha,\cdot)$, and driving Brownian motion $\wh\bm^k$, whose existence and uniqueness is guaranteed by Proposition \ref{prop:rdeu}, and define the process $\wh\X^{k,x_k}$ as in \eqref{eq:X}, but with $\wh\X^{k,x_k}$, $x_k$, $\wh\Z^{k,x_k}$ and $\wh\bm^k$ in place of $\X^{\alpha, x}$, $x$, $\Z^{\alpha, x}$ and $\bm$, respectively, and set $\wh\Y^{k,x_k}\doteq\wh\Z^{k,x_k}-\wh\X^{k,x_k}$. Since pathwise uniqueness implies uniqueness in law, we have 
	\be\label{eq:ZXYrhodistribution}(\wh\Z^{k,x_k},\wh\X^{k,x_k},\wh\Y^{k,x_k},\wh\rho_I^k)\buildrel{d}\over=(\Z^{x_k},\X^{x_k},\Y^{x_k},\rho_I^{x_k}),\ee
where $\buildrel{d}\over{=}$ indicates equality in distribution,  $\rho_I^{x_k}$ is the $\{\F_t\}$-stopping time defined in \eqref{eq:rhoIx} and $\wh\rho_I^k$ is the $\{\F_t^k\}$-stopping time defined as in \eqref{eq:rhoIx}, but with $Z^x$ replaced with $\wh \Z^k$. For each $k\in\N$, define the scaled processes $(\wt\Z^k,\wt\X^k,\wt\Y^k)$ as in \eqref{eq:Zk}--\eqref{eq:Yk}, respectively, but with $\wt\Z^k$, $\wt\X^k$, $\wt\Y^k$, $\wh\Z^{k,x_k}$, $\wh\X^{k,x_k}$, $\wh\Y^{k,x_k}$ and $\wh\rho_I^k$ in place of $\Z^k$, $\X^k$, $\Y^k$, $\Z^{x_k}$, $\X^{x_k}$, $\Y^{x_k}$ and $\rho_I^{x_k}$, respectively. Then by \eqref{eq:ZXYrhodistribution}, it follows that  
	\be\label{eq:distribution}(\wt\Z^k,\wt\X^k,\wt\Y^k)\buildrel d\over=(\Z^k,\X^k,\Y^k).\ee
In addition, by {the definitions above}, \eqref{eq:X}, \eqref{eq:Xk}, \eqref{eq:Zk} and \eqref{eq:wtbmk},  for $k\in\N$ and $t\geq0,$ 
\begin{align}\label{eq:wtXtk}
	\wt\X_t^k&=\frac{x_k-\bar{x}}{r_k}+\frac{1}{r_k}\int_0^{r_k^2t\wedge\wh\rho_I^k}\sigma(\alpha,\wh\Z_s^k)d\wh\bm_s^k=\frac{x_k-\bar{x}}{r_k}+\int_0^{t\wedge\frac{1}{r_k^2}\wh\rho_I^k}\sigma(\alpha,\bar{x}+r_k\wt\Z_s^k)d\bm_s,
\end{align}
where the final equality holds by {\eqref{eq:wtbmk}}, \eqref{eq:Zk} and the time-change theorem for stochastic integrals (see \cite[Chapter IV, Proposition 30.10]{Rogers2000a}).
\end{remark}

Let 
	$$\Pi_I:\R^J\mapsto\spaan\lb\{n_i,i\in I\}\rb$$ 
denote the orthogonal projection operator with respect to the usual Euclidean inner product $\ip{\cdot,\cdot}$. Since $\spaan\lb\{n_i,i\in I\}\rb$ is the orthogonal complement of $\{y-\bar x:y\in F_I\}$, 
	\be\label{eq:distxFI}|\Pi_I[x-\bar x]|=\dist(x,F_I),\qquad x\in G.\ee
Define the convex cone $G_{\bar{x}}$ as in \eqref{eq:Gx}, but with $\bar{x}$ in place of $x$, so that $G_{\bar{x}}$ is the domain of the ESP $\{(d_i(\alpha),n_i,0),i\in I\}$.

\begin{lem}\label{lem:rk2EwtZk}
For each $t<\infty$, $r_k^2\E\lsb\norm{\wt\Z^k}_t^2\rsb\to0$ as $k\to\infty$.
\end{lem}

\begin{proof}
Let $t<\infty$. By \eqref{eq:wtXtk}, the BDG inequalities, Tonelli's theorem and the Lipschitz continuity of $\sigma(\alpha,\cdot)$ implied by Assumption \ref{ass:drift},
\begin{align*}
	\E\lsb\norm{\wt\X^k}_t^2\rsb&\leq2\left|\frac{x_k-\bar{x}}{r_k}\right|^2+2C_2\int_0^t\E\lsb\norm{\sigma(\alpha,\bar{x}+r_k\wt\Z_s^k)}^2\rsb ds\\
	&\leq2\left|\frac{x_k-\bar{x}}{r_k}\right|^2+4C_2\norm{\sigma(\alpha,\bar x)}^2t+4C_2r_k^2\lip_{b,\sigma}^2\int_0^t\E\lsb\norm{\wt\Z^k}_s^2\rsb ds.
\end{align*}
By \eqref{eq:distribution} and Remark \ref{rmk:Zkesm}, $\wt \Z^k = \bar{\Gamma}^\alpha (\wt \X^k)$, and so the Lipschitz continuity of $\esm^\alpha$ stated in
Proposition \ref{prop:esmlip} implies 
	$\E[\norm{\wt\Z^k}_t^2]\leq(\lip_{\esm}(\alpha))^2\E[\norm{\wt\X^k}_t^2].$
      Combining this with  the last display and applying Grownwall's inequality yields
\begin{align*}
	\E\lsb\norm{\wt\Z^k}_t^2 \rsb\leq 2(\lip_{\esm}(\alpha))^2\lb\left|\frac{x_k-\bar{x}}{r_k}\right|^2+2C_2\norm{\sigma(\alpha,\bar x)}^2t\rb\exp\lb4C_2r_k^2(\lip_{\esm}(\alpha)\lip_{b,\sigma})^2t\rb.
\end{align*}
The lemma then follows from \eqref{eq:rk} and \eqref{eq:xkSrk}. 
\end{proof}

\begin{lem}\label{lem:ZkYkESPXk}
The data $\{(d_i(\cdot),n_i,0),i\in I\}$ satisfies Assumptions \ref{ass:setB} and \ref{ass:projection}. Hence, given $\alpha\in U$ and $\x\in\cts_{G_{\bar{x}}}(\R^J)$, there is a unique solution $(\z,\y)$ to the ESP $\{(d_i(\alpha),n_i,0),i\in I\}$ for $\x$. Furthermore, there exists $\lip_I(\alpha)<\infty$ such that given  $\x_j \in \cts_{G_{\bar{x}}}(\R^J)$ and the solution $(\z_j,\y_j)$ to the ESP $\{(d_i(\alpha),n_i,0),i\in I\}$ for $\x_j$, $j = 1, 2$, we have for $t<\infty$, 
	\be\label{eq:projlip}\norm{\Pi_I[\z_1]-\Pi_I[\z_2]}_t+\norm{\Pi_I[\y_1]-\Pi_I[\y_2]}_t\leq\lip_I(\alpha)\norm{\Pi_I[\x_1]-\Pi_I[\x_2]}_t.\ee
\end{lem}

The proof of Lemma \ref{lem:ZkYkESPXk} is given in Appendix \ref{apdx:ZkYkESPXk}.

\begin{proof}[Proof of Proposition \ref{prop:weakconvergence}]
By \eqref{eq:xkSrk} and \eqref{eq:SIr}, for each $k\in\N$,
	\be\label{eq:PiIZ0k}\frac{1}{r_k}|\Pi_I[x_k-\bar x]|=\frac{1}{r_k}\dist(x_k,F_I)=1.\ee
Therefore, by possibly taking a subsequence, also denoted $\{x_k\}_{k\in\N}$, there exists $x_\ast\in G_{\bar x}\cap\spaan(\{n_i,i\in I\})$ such that $|x_\ast|=1$ and
	\be\label{eq:PiIxkxast}\lim_{k\to\infty}\frac{1}{r_k}\Pi_I[x_k-\bar x]=\Pi_I[x_\ast]=x_\ast.\ee
Define
	\be\label{eq:Xast}\X_t^\ast\doteq x_\ast+\sigma(\alpha,\bar{x})\bm_t,\qquad t\geq0.\ee
Let $(\Z^\ast,\Y^\ast)$ denote the solution to the ESP $\{(d_i(\alpha),n_i,0),i\in I\}$ for $\X^\ast$, which is well defined by Lemma \ref{lem:ZkYkESPXk}. Let $t<\infty$. By \eqref{eq:wtXtk}, \eqref{eq:Xast}, the fact that $\Pi_I$ is a contraction operator, the BDG inequalities and the Lipschitz continuity of $\sigma(\alpha,\cdot)$,
\begin{align*}
	\E\lsb\norm{\Pi_I[\wt\X_t^k-\X_t^\ast]}_t^2\rsb&\leq3\left|\frac{1}{r_k}\Pi_I[x_k-\bar x]-x_\ast\right|^2\\
	&\qquad+3\E\lsb\sup_{0\leq s\leq t}\left|\int_0^{s\wedge\frac{1}{r_k^2}\wh\rho_I^k}\lb\sigma(\alpha,\bar{x}+r_k\wt\Z_u^k)-\sigma(\alpha,\bar{x})\rb d\bm_u\right|^2\rsb\\
	&\qquad+3\E\lsb\sup_{0\leq s\leq t}\left|\sigma(\alpha,\bar{x})\lb\bm_s-\bm_{s\wedge\frac{1}{r_k^2}\wh\rho_I^k}\rb\right|^2\rsb\\ 
	&\leq3\left|\frac{1}{r_k}\Pi_I[x_k-\bar x]-x_\ast\right|^2+3C_2\lip_{b,\sigma}^2tr_k^2\E\lsb\norm{\wt\Z^k}_t^2\rsb\\
	&\qquad+3C_2\norm{\sigma(\alpha,\bar x)}^2\lb t-\E\lsb t\wedge\frac{1}{r_k^2}\wh\rho_I^k\rsb\rb.
\end{align*}
The last display, along with \eqref{eq:PiIxkxast}, Lemma
\ref{lem:rk2EwtZk}, \eqref{eq:ZXYrhodistribution} and Lemma \ref{lem:rhok}, implies
	\be\label{eq:L2XkXast}\lim_{k\to\infty}\E\lsb\norm{\Pi_I[\wt\X^k-\X^\ast]}_t^2\rsb=0.\ee
Let $i\in I$ and $F_i^{\bar{x}}\doteq\{x\in G_{\bar{x}}:\ip{x,n_i}=0\}$ denote the $i$th face of the cone $G_{\bar{x}}$. Define 
	\be\label{eq:Sset}\mathcal{E}\doteq\lcb x\in G_{\bar{x}}:|\Pi_I[x]|\geq\frac{3}{4}\rcb.\ee
Since $G_{\bar{x}}$ is a convex cone, it follows that $\mathcal{E}$ is a connected set and $\mathcal{E}\cap F_i^{\bar{x}}$ is nonempty. In addition, since $|\Pi_I[x_\ast]|=|x_\ast|=1$, we have $x_\ast\in\mathcal{E}$. Therefore, we can define a continuous path $\x:[0,t]\mapsto\R^J$ such that  
\begin{itemize}
	\item[(a)] $\x(0)=x_\ast$, 
	\item[(b)] $\x(s)\in \mathcal{E}$ for all $s\in(0,t/2)$, 
	\item[(c)] $\x(t/2)$ lies in the relative interior of $F_i^{\bar{x}}$, and 
	\item[(d)] $\x(s)-\x(t/2)=-(s-t/2) d_i(\alpha)$ for all $s\in[t/2,t]$. 
\end{itemize}

Let $(\z,\y)$ denote the solution to the ESP $\{(d_i(\alpha),n_i,0),i\in I\}$ for $\x$ on $[0,t]$, whose existence and uniqueness is guaranteed by Lemma \ref{lem:ZkYkESPXk}. It is readily verified that the solution $(\z,\y)$ satisfies
\begin{align}\label{eq:zy0t2}
	(\z(s),\y(s))&=(\x(s),0)&&s\in[0,t/2], \\ \label{eq:zyt2t}
	(\z(s),\y(s))&=(\x(t/2),(s-t/2)d_i(\alpha))&&s\in[t/2,t]. 
\end{align}
Thus, by \eqref{eq:zy0t2}, \eqref{eq:zyt2t}, the continuity of $\x$ and the fact that $S$ is closed, $\z(s)\in S$ for all $s\in[0,t]$. Define $w:[0,t]\mapsto\R^K$ by 
	$$w(s)=\sigma^T(\alpha,\bar{x})a^{-1}(\alpha,\bar{x})(\x(s)-x_\ast),\qquad s\in[0,t],$$
where we recall $a(\alpha,\bar{x})\doteq\sigma(\alpha,\bar{x})\sigma^T(\alpha,\bar{x})$ is invertible due to Assumption \ref{ass:elliptic}. Then
	$$\x(s)=x_\ast+\sigma(\alpha,\bar{x})w(s),\qquad s\in[0,t].$$
Therefore, \eqref{eq:Xast} and the last display imply
\begin{align}\label{eq:Xastx}
	\norm{\X^\ast-\x}_t&\leq\norm{\sigma(\alpha,\bar{x})}\norm{\bm-w}_t.
\end{align}
Now, note that  $\Pi_I[d_i(\alpha)]\neq0$ holds because $i\in I$ and $\ip{d_i(\alpha),n_i}=1$. Together with (c), this implies that 
	\be\label{eq:vePiI}0<\ve<\frac{1}{4}\min\lcb1,|\Pi_I[d_i(\alpha)]|t\rcb \ee
such that 
	\be\label{eq:xt2njve}\ip{\x(t/2),n_j}>\ve,\qquad j\in I\setminus\{i\}.\ee

We now consider some implications of $\norm{\Pi_I[\Z^k]-\Pi_I[\z]}_t<\ve$ and $\norm{\Pi_I[\Y^k]-\Pi_I[\y]}_t<\ve$. 
The first set of implications, which are explained below, are as follows:
\begin{align}\label{eq:implication1}
	\norm{\Pi_I[\Z^k]-\Pi_I[\z]}_t<\ve\quad&\Rightarrow\quad|\ip{\Z_s^k-\z(s),n_j}|<\ve,&&j\in I,\;s\in[0,t],\\ \notag
	&\Rightarrow\quad\ip{\Z_s^k,n_j}\geq\ip{\x(t/2),n_j}-\ve,&&j\in I,\;s\in[t/2,t],\\ \notag
	&\Rightarrow\quad\ip{\Z_s^k,n_j}>0,&&j\in I\setminus\{i\},\;s\in[t/2,t].
\end{align}
The first implication holds because $\Pi_I$ is linear and self-adjoint, $\Pi_I[n_j]=n_j$ for all $j\in I$ and $\{n_j,j\in I\}$ are unit vectors. The second implication follows from the first relation and \eqref{eq:zyt2t}, and the final implication holds due to the second relation and \eqref{eq:xt2njve}. The next set of implications are as follows:  
\begin{align}\label{eq:implication2}
	\norm{\Pi_I[\Z^k]-\Pi_I[\z]}_t<\ve\quad&\Rightarrow\quad|\Pi_I[\Z_s^k]|>\frac{1}{2},\qquad s\in[0,t],\\ \notag
	&\Rightarrow\quad\sigma_I^k>t.
\end{align}
The first implication is due to the fact that $\z(s)\in S$ for all $s\in[0,t]$, \eqref{eq:Sset} and \eqref{eq:vePiI}. In turn, this implies the second implication due to the definition of $\sigma_I^k$ in \eqref{eq:sigmaIk1}. The third set of implications are as follows:
\begin{align}\label{eq:implication3}
	\norm{\Pi_I[\Y^k]-\Pi_I[\y]}_t<\ve\quad&\Rightarrow\quad|\Pi_I[\Y^k(t/2)]|<\ve, \mbox{ and } |\Pi_I[\Y^k(t)]|>\frac{t}{2}|\Pi_I[d_i(\alpha)]|-\ve,\\ \notag
	&\Rightarrow\quad\Y^k\text{ is nonconstant on }[t/2,t].
\end{align}
The first implication follows from \eqref{eq:zy0t2} and \eqref{eq:zyt2t}, and the second implication is due to \eqref{eq:vePiI}. Combining the implications \eqref{eq:implication1}--\eqref{eq:implication3}, we obtain the final set of implications:
\begin{align}\label{eq:implication4}
\norm{\Pi_I[\X^k]-\Pi_I[\x]}_{t}<\frac{\ve}{\lip_I(\alpha)}\quad&\Rightarrow\quad\norm{\Pi_I[\Z^k]-\Pi_I[\z]}_t<\ve,\;\norm{\Pi_I[\Y^k]-\Pi_I[\y]}_t<\ve,\\ \notag
	&\Rightarrow\quad\theta_i^k<\sigma_I^k,\qquad\forall\;k\geq k_0,  
\end{align}
where $\lip_I(\alpha)$ is the constant in \eqref{eq:projlip}. The first implication follows from Lemma \ref{lem:ZkYkESPXk} and because  $(Z^k, Y^k)$ is a solution to the ESP $\{(d_i(\alpha),n_i,0),i\in I\}$ for $X^k$ by Remark \ref{rmk:Zkesm} and $(\z,\y)$ is a solution to the same ESP for $\x$ by construction. The second implication uses \eqref{eq:implication3}, \eqref{eq:implication1} and the fact that $\Y^k$ can only increase when $\Z^k$ lies on the boundary $\partial G_{\bar{x}}$ to conclude that $\theta_i^k\leq t$, which along with \eqref{eq:implication2} yields $\theta_i^k<\sigma_I^k$. 

Now, by \eqref{eq:thetaik1}, \eqref{eq:sigmaIk1}, \eqref{eq:implication4}, \eqref{eq:distribution}, the fact that $\Pi_I$ is a linear contraction operator, \eqref{eq:Xastx}, the relations $\norm{\sigma(\alpha,\bar x)}>0$ (due to Assumption \ref{ass:elliptic}) and $\P(A\cap B)\geq\P(A)-\P(B^c)$ for $A,B\in\F$, and Chebyshev's inequality, we have
\begin{align*}
	\P\lb\theta_i^k<\sigma_I^k\rb&\geq\P\lb\norm{\Pi_I[\X^k]-\Pi_I[\x]}_{t}<\frac{\ve}{\lip_I(\alpha)}\rb\\
	&\geq\P\lb\norm{\Pi_I[\wt\X^k]-\Pi_I[\X^\ast]}_{t}<\frac{\ve}{2\lip_I(\alpha)},\;\norm{\X^\ast-\x}_{t}<\frac{\ve}{2\lip_I(\alpha)}\rb\\
	&\geq\P\lb\norm{\Pi_I[\wt\X^k]-\Pi_I[\X^\ast]}_{t}<\frac{\ve}{2\lip_I(\alpha)},\;\norm{\bm-w}_{t}<\frac{\ve}{2\lip_I(\alpha)\norm{\sigma(\alpha,\bar{x})}}\rb\\
	&\geq\P\lb\norm{\bm-w}_{t}<\frac{\ve}{2\lip_I(\alpha)\norm{\sigma(\alpha,\bar{x})}}\rb-\P\lb\norm{\Pi_I[\wt\X^k]-\Pi_I[\X^\ast]}_{t}\geq\frac{\ve}{2\lip_I(\alpha)}\rb\\
	&\geq\P\lb\norm{\bm-w}_{t}<\frac{\ve}{2\lip_I(\alpha)\norm{\sigma(\alpha,\bar{x})}}\rb-\frac{4(\lip_I(\alpha))^2}{\ve^2}\E\lsb\norm{\Pi_I[\wt\X^k]-\Pi_I[\X^\ast]}_{t}^2\rsb.
\end{align*}
Taking limits as $k\to\infty$ in the last display and using \eqref{eq:L2XkXast} yields
$$\liminf_{k\to\infty}\P\lb\theta_i^k<\sigma_I^k\rb\geq\P\lb\norm{\bm-w}_t<\frac{\ve}{2\lip_I(\alpha)\norm{\sigma(\alpha,\bar{x})}}\rb>0,$$ 
where the final inequality is due to the fact that $K$-dimensional
Wiener measure assigns positive measure to (relatively) open subsets of $\{v\in\cts(\R^K):v(0)=0\}$ (see, e.g., \cite[Lemma 3.1]{Stroock1972}). This proves \eqref{eq:sigmaIxksigmaI}.

The proof of \eqref{eq:tauIxktauI} follows an argument analogous to the one used to prove \eqref{eq:sigmaIxksigmaI}. The main difference is to define $S\doteq\{ x\in G_{\bar{x}}:|\Pi_I[x]|\leq\frac{5}{4}|\Pi_I[x_\ast]|\}$ and to use $\tau_I^k$ in place of $\sigma_I^k$. To avoid repetition, we omit the details. 
\end{proof}

\subsection{Verifications of conditions 3 and 4}\label{sec:jitter34}

We first verify conditions 3 and 4 of the boundary jitter property when the drift coefficient satisfies $b(\alpha,\cdot)\equiv0$. In the proof of Theorem \ref{thm:jitter} below, we use a change of measure argument to verify the conditions for general Lipschitz continuous drift coefficients.

Given $x\in G$, $s\geq0$, $i\in\allN$, $I\in\J$ and $C<\infty$, define the $\{\F_t\}$-stopping times 
\begin{align}\label{eq:thetaisx}
	\theta_i^{s,x}&\doteq\inf\{t>s:\Z_t^x\in F_i\},\\ \label{eq:thetaIsx}
	\theta_I^{s,x}&\doteq\inf\{t>s:\Z_t^x\in F_I\}, \\ \label{eq:rhoIsx}
	\rho_I^{s,x}&\doteq\inf\{t>s:\Z_t^x\in\cup_{j\in\allN\setminus I}F_j\}, \\ \label{eq:rhoMsx}
	\xi_C^{s,x}&\doteq\inf\{t>s:|\Z_t^x|\geq C\}.
\end{align}
When $s=0$ we omit the ``$s$'' superscript and write $\theta_i^x$, $\theta_I^x$, $\rho_I^x$ and $\xi_C^x$ for $\theta_i^{0,x}$, $\theta_I^{0,x}$, $\rho_I^{0,x}$ and $\xi_C^{0,x}$, respectively. Note that the definitions of $\theta_i^x$ and $\rho_I^x$ here coincide with the ones given in \eqref{eq:thetaix} and \eqref{eq:rhoIx}, respectively.

\begin{lem}\label{lem:jitter3prelim}
Let $I\in\J$, $T<\infty$ and $C<\infty$. For each $i\in I$ there exists $\ve\in(0,1)$ and $k_0\in\N$ such that for each $k\geq k_0$,
\begin{align}\label{eq:rkp}
	\P(\theta_i^x\wedge\rho_I^x\wedge\xi_C^x\wedge T<\sigma_I^x)&\geq \ve\qquad\text{for all }x\in\mathbb{S}_I(r_k)\\ \label{eq:thetaixtaux}
	\P(\theta_i^x\wedge\rho_I^x\wedge\xi_C^x\wedge T<\tau_I^x)&\geq\ve\qquad\text{for all }x\in\mathbb{S}_I(r_k).
\end{align}
\end{lem}

\begin{proof}
Let $i\in I$. We first prove \eqref{eq:rkp}. For a proof by contradiction, recalling from \eqref{eq:sigmaIxksigmaI} that $\lim_{k\to\infty}\P(0 < \rho_I^{x_k}) > 0$, suppose there is a sequence $\{x_k\}_{k\in\N}$ in $G$ such that $x_k\in\mathbb{S}_I(r_k)$ and $|x_k|<C$ for each $k\in\N$, and
	\be\label{eq:jitter3contradiction}\lim_{k\to\infty}\P(\theta_i^{x_k}\wedge\rho_I^{x_k}\wedge\xi_C^{x_k}\wedge T<\sigma_I^{x_k})=0.\ee
Since $|x_k|\leq C$ and $\dist(x_k,F_I)=r_k$ for all $k\in\N$ due to \eqref{eq:SIr}, by possibly taking a subsequence $\{k_\ell\}_{\ell\in\N}$, we can assume there exists $\bar{x}\in F_I$ such that $x_{k_\ell}\to\bar x$ as $\ell\to\infty$. Then by \eqref{eq:sigmaIxksigmaI},
	$$\liminf_{\ell\to\infty}\P(\theta_i^{x_{k_\ell}}\wedge\rho_I^{x_{k_\ell}}\wedge\xi_C^{x_{k_\ell}}\wedge T<\sigma_I^{x_{k_\ell}})\geq\liminf_{\ell\to\infty}\P(\theta_i^{{x_{k_\ell}}}<\sigma_I^{{x_{k_\ell}}})>0,$$
which contradicts {\eqref{eq:jitter3contradiction}}.  With this contradiction thus obtained, it follows {that} there exist $\ve\in(0,1)$ and $k_0\in\N$ such that for each $k\geq k_0$, \eqref{eq:rkp} holds. The proof of \eqref{eq:thetaixtaux} is exactly analogous to the proof of \eqref{eq:rkp}, except it uses $\tau_I^{x_k}$, $\tau_I^k$ and \eqref{eq:tauIxktauI} in place of $\sigma_I^{x_k}$, $\sigma_I^k$ and \eqref{eq:sigmaIxksigmaI}, respectively, so we omit the details. 
\end{proof}

\begin{lem}\label{lem:ZtGcirc}
{$\P(\Z_t^x\in G^\circ)=1$ for all $t>0$.}
\end{lem}
\begin{proof}
{The proof of Lemma \ref{lem:ZtGcirc} relies on Assumption \ref{ass:elliptic} and \eqref{eq:jitter2}. Since it can be established 
in a manner exactly analogous to the proof of \cite[equation (A.4)]{Budhiraja2007}, which establishes the claim when the set $\V^\alpha$ defined in \eqref{eq:Valpha} is empty,  we omit the details.} 
\end{proof}

\begin{lem}\label{lem:jitter3}
Suppose Assumption \ref{ass:elliptic} holds, $b(\alpha,\cdot)\equiv0$
and $T<\infty$. 
Then for each $x\in G$ a.s.\ $\Z^x$ satisfies condition 3 of the boundary jitter property on $[0,T]$.
\end{lem} 
\begin{proof}
Let $x\in G$ and $T<\infty$.  {Due to} Definition \ref{def:jitter}, the upper semicontinuity of $\allN(\cdot)$ (Lemma \ref{lem:allNusc}) and the continuity of $\Z^x$, we have
\begin{align*}
	&\lcb\Z^x\text{ satisfies condition 3 of the boundary jitter property on }[0,T]\rcb\\
	&\qquad=\bigcap_{I\in\J}\bigcap_{i\in I}\bigcap_{s\in\Q\cap(0,T)}\lcb\theta_i^{s,x}\wedge T<\theta_I^{s,x}\rcb.
\end{align*}
We claim that
	\be\label{eq:bigcaprhoIsx}\bigcap_{I\in\J}\bigcap_{i\in I}\bigcap_{s\in\Q\cap(0,T)}\lcb\theta_i^{s,x}\wedge T<\theta_I^{s,x}\rcb=\bigcap_{I\in\J}\bigcap_{i\in I}\bigcap_{s\in\Q\cap(0,T)}\lcb\theta_i^{s,x}\wedge\rho_I^{s,x}\wedge T<\theta_I^{s,x}\rcb.\ee
The left-hand side of \eqref{eq:bigcaprhoIsx} is clearly contained in the right-hand side. Thus, to prove the claim it suffices to show that for any given $I\in\J$, $i\in I$ and $s\in\Q\cap(0,T)$,
	$$\bigcap_{r\in\Q\cap(0,T)}\lcb\theta_i^{r,x}\wedge\rho_I^{r,x}\wedge T<\theta_I^{r,x}\rcb\subseteq\lcb\theta_i^{s,x}\wedge T<\theta_I^{s,x}\rcb,$$
or equivalently,
	\be\label{eq:thetaisxTgeqthetaIsx}\lcb\theta_i^{s,x}\wedge T\geq\theta_I^{s,x}\rcb\subseteq\bigcup_{r\in\Q\cap(0,T)}\lcb\theta_i^{r,x}\wedge\rho_I^{r,x}\wedge T\geq\theta_I^{r,x}\rcb.\ee	
{Fix $I\in\J$, $i\in I$, $s\in\Q\cap(0,T)$ and $\omega\in\lcb\theta_i^{s,x}\wedge T\geq\theta_I^{s,x}\rcb$.}  If $\theta_I^{s,x}(\omega)=s$, then $\omega\in\lcb\theta_i^{s,x}\wedge\rho_I^{s,x}\wedge T\geq\theta_I^{s,x}\rcb$, so \eqref{eq:thetaisxTgeqthetaIsx} holds. Suppose $\theta_I^{s,x}(\omega)>s$. By the upper semicontinuity of $\allN(\cdot)$, the continuity of $\Z^x$ and the definition of $\theta_I^{s,x}$ in \eqref{eq:thetaIsx}, there exists $r\in(s,\theta_I^{s,x}(\omega))\cap\Q$ such that $\allN(\Z_u^x)\subseteq I$ for all $u\in[r,\theta_I^{s,x}(\omega)]$. Thus, by the definition of $\rho_I^{r,x}$ in \eqref{eq:rhoIsx}, $\omega\in\{\theta_i^{r,x}\wedge\rho_I^{r,x}\wedge T\geq\theta_I^{r,x}\}$. This proves \eqref{eq:thetaisxTgeqthetaIsx} and so the claim \eqref{eq:bigcaprhoIsx} holds.  

To show \eqref{eq:bigcaprhoIsx}, it clearly suffices to show that 
$\P(\theta_i^{s,x}\wedge\rho_I^{s,x}\wedge T<\theta_I^{s,x})=1$ for each $I\in\J$, $i\in I$ and $s\in\Q\cap(0,T)$. Fix $I\in\J$, $i\in I$ and $s\in\Q\cap(0,T)$. Using the Markov property of $\Z^x$ and Lemma \ref{lem:ZtGcirc}, we have
	\be\label{eq:thetaIythetaiy}\P(\theta_i^{s,x}\wedge\rho_I^{s,x}\wedge T<\theta_I^{s,x})=\int_{G^\circ}\P(\theta_i^y\wedge\rho_I^y\wedge T<\theta_I^y)\P(\Z_s^x\in dy).\ee
Hence, we are left to show that $\P(\theta_i^y\wedge\rho_I^y\wedge T<\theta_I^y)=1$ for all $y\in G^\circ$. Since a.s.\ $\xi_C^x\to\infty$ as $C\to\infty$, it is enough to show that for all $y\in G^\circ$ and $C<\infty$,
	\be\label{eq:thetaiythetaIy1}\P(\theta_i^y\wedge\rho_I^y\wedge\xi_C^y\wedge T<\theta_I^y)=1.\ee

Fix $y\in G^\circ$ and $C<\infty$. Let $\{r_k\}_{k\in\N}$ be the decreasing sequence defined in \eqref{eq:rk} and let $\ve\in(0,1)$ and $k_0 \in \mathbb{N}$ be such that \eqref{eq:rkp} holds for all $k \geq k_0$. For ${k\geq k_0}$ and $z\in G$, let $\sigma_I^z(r_k)\doteq\inf\{t>0:\dist(\Z_t^z,F_I)\leq r_k\}$. The definition of $\sigma_I^x$ in \eqref{eq:sigmaIx} implies that for each $k\in\N$, $\sigma_I^z(r_{k+1})=\sigma_I^z$ for all $z\in\mathbb{S}_I(r_k)$. Then by the continuity of the sample paths of $\Z^y$, the strong Markov property of $\Z^y$ and \eqref{eq:rkp},
\begin{align*}
	\P(\theta_I^y\leq\theta_i^y\wedge\rho_I^y\wedge\xi_C^y\wedge T)&=\P(\sigma_I^y(r_k)\leq\theta_i^y\wedge\rho_I^y\wedge\xi_C^y\wedge T\text{ for all }k\geq 1)\\
	&\leq\lim_{k\to\infty}\prod_{j=1}^k\sup_{z\in\mathbb{S}_I(r_j)}\P(\sigma_I^z\leq\theta_i^z\wedge\rho_I^z\wedge\xi_C^z\wedge T)\\
	&\leq\lim_{k\to\infty}(1-\ve)^k,
\end{align*}
which is equal to zero. This proves \eqref{eq:thetaiythetaIy1} holds.
\end{proof}

\begin{lem}\label{lem:jitter4}
Suppose Assumption \ref{ass:elliptic} holds and $b(\alpha,\cdot)\equiv0$. Then for each $x\in G$, a.s.\ $\Z^x$ satisfies condition 4 of the boundary jitter property.
\end{lem}

\begin{proof}
If $x\not\in\U$ condition 4 holds trivially by Definition \ref{def:jitter}. Fix $x\in\U$ and let $i\in\allN(x)$. Set $I\doteq\allN(x)$. Let $\{r_k\}_{k\in\N}$ in $(0,\infty)$, $k_0 \in \N$ and $\ve\in(0,1)$ be such that \eqref{eq:thetaixtaux} holds for all $k \geq k_0$. For $k\geq k_0$ and $y\in G$, let $\tau_I^y(r_k)\doteq\inf\{t>0:\dist(\Z_t^y,F_I)\geq r_k\}$. Then by \eqref{eq:tauIx}, $\tau_I^y(r_{k-1})=\tau_I^y$ for $y\in\mathbb{S}_I(r_k)$. Let $T<\infty$. By the continuity of the sample paths of $\Z^x$, the strong Markov property of $\Z^x$ and \eqref{eq:thetaixtaux}, for each ${k^\dagger \geq k_0}$, 
\begin{align*}
	\P\lb\tau_I^x({r_{k^\dagger}})\leq\theta_i^x\wedge\rho_I^x\wedge\xi_C^x\wedge T\rb&=\P\lb\tau_I^x(r_k)\leq\theta_i^x\wedge\rho_I^x\wedge\xi_C^x\wedge T\text{ for all $k\geq{k^\dagger}$}\rb\\
	&\leq\lim_{K\to\infty}\prod_{k={k^\dagger}+1}^K\sup_{y\in\mathbb{S}_I(r_k)}\P\lb\tau_I^y\leq\theta_i^y\wedge\rho_I^y\wedge\xi_C^y\wedge T\rb\\
	&\leq\lim_{K\to\infty}(1-\ve)^{K-{k^\dagger}}.
\end{align*}
Upon sending $C\to\infty$ and $T\to\infty$, this proves $\P(\theta_i^x\wedge\rho_I^x < \tau_I^x (r_k)\;\forall\;{k \geq k_0}) = 1$. Since a.s.\ $\Z^x$ spends zero Lebesgue time on the boundary by Lemma \ref{lem:jitter2}, it follows that a.s.\ for every $\delta>0$, there exists $t\in(0,\delta)$ such that $\dist(\Z_t^x,F_I)>0$. In particular, this implies that a.s.\ ${\tau^x_I(r)} \downarrow0$ as $r\downarrow0$. Consequently, $\P(\theta_i^x\wedge\rho_I^x=0)=1$. By the upper semicontinuity of $\allN(\cdot)$ (Lemma \ref{lem:allNusc}) and the continuity of $\Z^x$, a.s.\ $\rho_I^x>0$. Thus, $\P(\theta_i^x\wedge\rho_I^x=0)=1$ implies that $\P(\theta_i^x=0)=1$. Since $i\in\allN(x)$ was arbitrary and $\allN(x)$ is a finite set, we have $\P(\theta_i^x=0\;\forall\;i\in\allN(x))=1$. Along with Lemma \ref{lem:jitter3}, this implies the set 
	\be\label{eq:cond3thetai}\{\Z^x\text{ satisfies condition 3 of the boundary jitter property}\}\cap\{\theta_i^x=0\;\forall\;i\in\allN(x)\},\ee
has $\P$-measure one. Let $\omega$ belong to the set \eqref{eq:cond3thetai}. Let $i\in\allN(x)$ and $\delta>0$. Since $\theta_i^x(\omega)=0$ there exists $t\in(0,\delta)$ such that $\Z_t^x(\omega)\in F_i$. If $\Z_t^x(\omega)\in\S$, then $\allN(\Z_t^x(\omega))=\{i\}$. On the other hand, if $\Z_t^x(\omega)\in\U$, then condition 3 of the boundary jitter property implies there exists $s\in(0,t)$ such that $\allN(\Z_s^x(\omega))=\{i\}$. Since this holds for all $i\in\allN(x)$ and $\delta>0$, $\Z^x(\omega)$ satisfies condition 4 of the boundary jitter property.
\end{proof}

\begin{proof}[Proof of Theorem \ref{thm:jitter}]
Condition 2 of the boundary jitter property follow immediately from
Lemma \ref{lem:jitter2}. When $b(\alpha,\cdot)\equiv0$, conditions 1, 3 and
4 follow from {Lemmas \ref{lem:jitter1}}, \ref{lem:jitter3} and
\ref{lem:jitter4}, respectively. To see that conditions 1, 3 and 4 hold under general Lipschitz continuous drift coefficients, we use a change of measure argument. Since the filtration $\{\F_t\}$ is right-continuous, we see that 
	\be\label{eq:jitter4F0}\{\Z^x\text{ satisfies condition 4 of the boundary jitter property}\}\in\F_0.\ee
In addition, it is readily verified that for $T<\infty$,
	\be\label{eq:jitter3FT}\{\Z^x\text{ satisfies conditions 1 and 3 of the boundary jitter property for all }t\in[0,T]\}\in\F_T,\ee
and $\Z^x$ a.s.\ satisfies conditions 1 and 3 of the boundary jitter
property if and only if for each $T<\infty$, the event  {in} \eqref{eq:jitter3FT} has $\P$-measure one. Let $T<\infty$. Observe that the uniform ellipticity of $a(\alpha,\cdot)\doteq\sigma(\alpha,\cdot)\sigma^T(\alpha,\cdot)$ stated in Assumption \ref{ass:elliptic} ensures that $a^{-1}(\alpha,\cdot)$ exists, and define
	$$\wh\bm_t\doteq\bm_t-\int_0^t\sigma^T(\alpha,\Z_s^x)a^{-1}(\alpha,\Z_s^x)b(\alpha,\Z_s^x)ds,\qquad t\in[0,T].$$
By a standard argument using the Lipschitz continuity of $b(\alpha,\cdot)$ and Girsanov's transformation (see, e.g., the proof of \cite[Theorem 4.1]{Kang2010}), there is a probability measure $\wt\P$ on $(\Omega,\F_T)$ equivalent to $\P$ such that under $\wt\P$, $\{\wh\bm_t,t\in[0,T]\}$ is a Brownian motion on $(\Omega,\F_T,\wt\P)$. Substituting $\wh\bm$ into \eqref{eq:Z}, we see that
	$$\Z_t^x=x+\int_0^t\sigma(\alpha,\Z_s^x)d\wh\bm_s+\Y_t^x,\qquad t\in[0,T].$$
By Lemmas \ref{lem:jitter3} and \ref{lem:jitter4} and because $\wt\P$ and $\P$ are equivalent on $(\Omega,\F_T)$, the events \eqref{eq:jitter4F0} and \eqref{eq:jitter3FT} have $\P$-measure one. Since this holds for all $T<\infty$, the proof is complete.
\end{proof}

\section{The derivative process}\label{sec:derivativeprocessproof}

In this section we prove Theorem \ref{thm:dpunique}, which establishes pathwise uniqueness of a derivative process along the reflected
diffusion $\Z^{\alpha,x}$. In Section \ref{sec:dp}, we describe the relationship between the derivative process and an associated deterministic problem, 
called the derivative problem. In Section \ref{sec:pathunique} we prove pathwise uniqueness and provide conditions for strong existence of the derivative process.

\subsection{The derivative problem}\label{sec:dp}

The derivative problem was first introduced in \cite[Definition
3.4]{Lipshutz2016} as an axiomatic framework for characterizing
directional derivatives of the ESM. 

\begin{defn}\label{def:dp}
Let $\alpha\in U$. Suppose $(\z,\y)$ is a solution to the ESP $\{(d_i(\alpha),n_i,c_i),i\in\allN\}$ for $\x\in\cts_G(\R^J)$. Let $\psi\in\dr(\R^J)$. Then $(\phi,\eta)\in\dr(\R^J)\times\dr(\R^J)$ is a solution to the derivative problem (associated with $\{(d_i(\alpha),n_i,c_i)\}$) along $\z$ for $\psi$ if $\eta(0)\in\spaan[d(\alpha,\z(0))]$ and if for all $t\geq0$, the following conditions hold:
\begin{itemize}
	\item[1.] $\phi(t)=\psi(t)+\eta(t)\in H_{\z(t)}$;
	\item[2.] $\phi(t)\in H_{\z(t)}$;
	\item[3.] for all $s\in[0,t)$, 
		$$\eta(t)-\eta(s)\in\spaan\lsb\cup_{u\in(s,t]}d(\alpha,\z(u))\rsb.$$
\end{itemize}
If there exists a unique solution $(\phi,\eta)$ to the derivative problem along $\z$ for $\psi$, we write $\phi=\dm_\z^\alpha[\psi]$ and refer to $\dm_\z^\alpha$ as the derivative map along with $\z$.
\end{defn}

The derivative problem can be viewed as a linearization of the ESP along a given solution $(h,g)$ of the ESP (compare Definition \ref{def:dp} with Definition \ref{def:esp}).

\begin{remark}\label{rmk:dpde} 
Given $\alpha\in U$, $x\in G$ and a derivative process $\phib^{\alpha,x}$ along $\Z^{\alpha,x}$, let $\psib^{\alpha,x}=\{\psib_t^{\alpha,x},t\geq0\}$ be the continuous $\{\F_t\}$-adapted process taking values in $\lin(\R^M\times H_x,\R^J)$ defined, for all $t\geq0$ and $(\beta,y)\in\R^M\times H_x$, by
\begin{align}\label{eq:psib}
	\psib_t^{\alpha,x}[\beta,y]&=y+\int_0^tb'(\alpha,\Z_s^{\alpha,x})[\beta,\phib_s^{\alpha,x}[\beta,y]]ds+\int_0^t\sigma'(\alpha,\Z_s^{\alpha,x})[\beta,\phib_s^{\alpha,x}[\beta,y]]d\bm_s\\ \notag
	&\qquad+R'(\alpha)[\beta]\L_t^{\alpha,x},
\end{align}
where $\{R'(\alpha)[\beta]\L_t^{\alpha,x},t\geq0\}$ is the process defined in Remark \ref{rmk:DRL}. Then by the properties stated in Definition \ref{def:de} and the statement of the derivative problem in Definition \ref{def:dp}, a.s.\ for all $(\beta,y)\in\R^M\times H_x$, $(\phib^{\alpha,x}[\beta,y],\etab^{\alpha,x}[\beta,y])$ is a solution to the derivative problem along $\Z^{\alpha,x}$ for $\psib^{\alpha,x}[\beta,y]$.
\end{remark}

The following Lipschitz continuity property of the derivative map was established in \cite{Lipshutz2016}.

\begin{prop}[{\cite[Theorem 5.4]{Lipshutz2016}}]
\label{prop:dmlip}
Let $\alpha\in U$. There exists $\lip_\dm(\alpha)<\infty$ such that if $(\z,\y)$ is a solution to the ESP $\{(d_i(\alpha),n_i,c_i),i\in\allN\}$ for $\x$, $(\phi_1,\eta_1)$ is a solution to the derivative problem along $\z$ for $\psi_1\in\cts(\R^J)$, and $(\phi_2,\eta_2)$ is a solution to the derivative problem along $\z$ for $\psi_2\in\cts(\R^J)$, then for all $t<\infty$,
	\be
\label{DMLC}
\norm{\phi_1-\phi_2}_t\leq\lip_\dm(\alpha)\norm{\psi_1-\psi_2}_t.\ee
\end{prop}

Note that the Lipschitz constant in \eqref{DMLC} depends only on $\alpha\in U$, and not on $h \in \cts_G(\R^J)$. 

\subsection{Pathwise uniqueness of the derivative process}\label{sec:pathunique}

\begin{proof}[Proof of Theorem \ref{thm:dpunique}]
Let $(\beta,y)\in\R^M\times H_x$. According to Remark \ref{rmk:dpde}, 
\begin{align}\label{eq:phibdmpsib}
	\phib^{\alpha,x}[\beta,y]=\dm_{\Z^{\alpha,x}}^\alpha[\psib^{\alpha,x}[\beta,y]]\qquad\text{and}\qquad\wt\phib^{\alpha,x}[\beta,y]=\dm_{\Z^{\alpha,x}}^\alpha[\wt\psib^{\alpha,x}[\beta,y]],
\end{align}
where $\psib^{\alpha,x}$ is defined as in \eqref{eq:psib} and $\wt\psib^{\alpha,x}$ is defined analogously, but with $\wt\phib^{\alpha,x}$ and $\wt\psib^{\alpha,x}$ in place of $\phib^{\alpha,x}$ and $\psib^{\alpha,x}$, respectively. By \eqref{eq:psib}, the Cauchy-Schwarz inequality, the BDG inequalities, the bounds on $\norm{b'(\alpha,x)}$ and $\norm{\sigma'(\alpha,x)}$ stated in Assumption \ref{ass:drift} and Tonelli's theorem, we have, for $t\geq0$,
\begin{align*}
	\E\lsb\norm{\psib^{\alpha,x}[\beta,y]-\wt\psib^{\alpha,x}[\beta,y]}_t^2\rsb&\leq2\E\lsb\sup_{s\in[0,t]}\left|\int_0^sb'(\alpha,\Z_u^{\alpha,x})[0,\phib^{\alpha,x}[\beta,y]-\wt\phib^{\alpha,x}[\beta,y]]du\right|^2\rsb\\
	&\qquad+2\E\lsb\sup_{s\in[0,t]}\left|\int_0^s\sigma'(\alpha,\Z_u^{\alpha,x})[0,\phib^{\alpha,x}[\beta,y]-\wt\phib^{\alpha,x}[\beta,y]]dW_u\right|^2\rsb\\
	&\leq2\lip_{b,\sigma}^2(t+C_2)\int_0^t\E\lsb\norm{\phib^{\alpha,x}[\beta,y]-\wt\phib^{\alpha,x}[\beta,y]}_s^2\rsb ds.
\end{align*}
Using \eqref{eq:phibdmpsib}, the Lipschitz continuity of the derivative map shown in Proposition \ref{prop:dmlip} and applying Gronwall's inequality, we obtain,
	$$\E\lsb\norm{\phib^{\alpha,x}[\beta,y]-\wt\phib^{\alpha,x}[\beta,y]}_t^2\rsb=0.$$
Since $t\geq0$ and $(\beta,y)\in\R^M\times H_x$ were arbitrary, and both $\phib^{\alpha,x}$ and $\wt\phib^{\alpha,x}$ are linear functions of $(\beta,y)\in\R^M\times H_x$, this proves that a.s.\ $\phib^{\alpha,x}=\wt\phib^{\alpha,x}$.
\end{proof}

\section{Pathwise differentiability of reflected diffusions}\label{sec:pathwiseproof}

In Section \ref{sec:SMderivatives} we recall the definition and characterization of a directional derivative of the ESM $\esm^\alpha$ from \cite{Lipshutz2016}. In Sections \ref{subs:pfpathwise} and \ref{sec:psivepsi} we use properties of these directional derivatives to characterize pathwise derivatives of a reflected diffusion in terms of derivative processes. 

\subsection{Directional derivatives of the ESM}\label{sec:SMderivatives}

Fix $\alpha\in U$. Recall the definition of the ESP $\{(d_i(\alpha),n_i,c_i),i\in\allN\}$ given in Definition \ref{def:esp}. By Proposition \ref{prop:esp}, the associated ESM $\esm^\alpha$ is well defined on $\cts_G(\R^J)$. We now introduce the notion of a directional derivative of $\esm^\alpha$. For $\x\in\cts_G(\R^J)$, $\psi\in\cts(\R^J)$ and $\ve>0$, define $\nabla_\psi^\ve\esm(\x)\in\cts(\R^J)$ by
	\be\label{eq:nablavepsiesm}\nabla_\psi^\ve\esm^\alpha(\x)\doteq\frac{\esm^\alpha(\x+\ve\psi)-\esm^\alpha(\x)}{\ve}.\ee

\begin{defn}
Given $\x\in\cts_G(\R^J)$ and $\psi\in\cts(\R^J)$, the directional derivative of $\esm^\alpha$ evaluated at $\x$ in the direction $\psi$ is a function $\nabla_\psi\esm^\alpha(\x)$ from $[0,\infty)$ into $\R^J$ defined as the pointwise limit
	\be\label{eq:nablapsiesm}\nabla_\psi\esm^\alpha(\x)(t)\doteq\lim_{\ve\downarrow0}\nabla_\psi^\ve\esm^\alpha(\x)(t),\qquad t\geq0.\ee
\end{defn}

\begin{prop}[{\cite[Proposition 2.17]{Lipshutz2016}}]
\label{prop:esmpsive}
Given $\x\in\cts_G(\R^J)$ and $\psi\in\cts(\R^J)$ such that $\nabla_\psi\esm^\alpha(\x)$ exists, suppose $\{\psi_\ve\}_{\ve>0}$ is a family in $\cts(\R^J)$ such that $\psi^\ve\to\psi$ in $\cts(\R^J)$ as $\ve\downarrow0$. Then 
	$$\lim_{\ve\downarrow0}\nabla_{\psi^\ve}^\ve\esm^\alpha(\x)(t)=\nabla_\psi\esm^\alpha(\x)(t),\qquad t\geq0.$$
\end{prop}

\begin{prop}\label{prop:nablaesmlip}
Given $\x,\psi,\wt\psi\in\cts(\R^J)$, suppose $\nabla_\psi\esm^\alpha(\x)$ and $\nabla_{\wt\psi}\esm^\alpha(\x)$ exist. Then for all $t<\infty$,
	\be\label{eq:nablaesmlip}\norm{\nabla_\psi\esm(\x)-\nabla_{\wt\psi}\esm(\x)}_t\leq\lip_{\esm}(\alpha)\norm{\psi-\wt\psi}_t.\ee
\end{prop}

\begin{proof}
Let $t<\infty$ and $s\in[0,t]$. By \eqref{eq:nablapsiesm} and the Lipschitz continuity of the ESM stated in Proposition \ref{prop:esmlip},
\begin{align*}
	|\nabla_\psi\esm^\alpha(\x)(s)-\nabla_{\wt\psi}\esm^\alpha(\x)(s)|=\lim_{\ve\downarrow0}\ve^{-1}|\esm^\alpha(\x+\ve\psi)(s)-\esm^\alpha(\x+\ve\wt\psi)(s)|\leq\lip_{\esm}(\alpha)\norm{\psi-\wt\psi}_s.
\end{align*}
Taking suprema over $s\in[0,t]$ of both sides of the last display yields \eqref{eq:nablaesmlip}.
\end{proof}

The notion of directional derivatives of the one-dimensional Skorokhod map was first introduced in \cite{Mandelbaum1995} (see also \cite[Corollary 9.5.1]{Whitt2002} and \cite[Theorem 3.2]{ManRam10}) to prove a diffusion approximation of a time-inhomogeneous queue. Directional derivatives of ESMs in the orthant with reflection matrices that are $\mathcal{M}$-matrices were subsequently studied in \cite{ManRam10}. The result in \cite{ManRam10} covers a large class of ESPs of interest, including those arising in rank-based models \cite{Banner2005,Ichiba2011} and interacting particle systems \cite{Burdzy2002,Warren2007,Zam2004}, but does not include many others arising in applications, such as multiclass feedforward queueing networks (see, e.g., \cite{Chen2000}). In \cite{Lipshutz2016} directional derivatives of a much broader class of ESMs in polyhedral domains were characterized when the solution to the ESM satisfies the boundary jitter property. We now recall the main result in \cite{Lipshutz2016}. Recall the derivative map introduced in Definition \ref{def:dp}, and the set $\W^\alpha$ defined in \eqref{eq:Walpha}.

\begin{prop}[{\cite[Theorem 3.12]{Lipshutz2016}}]
\label{prop:smderivative}
Given $\x\in\cts_G$, let $(\z,\y)$ denote the solution to the ESP for $\x$. Suppose $\z(t)\not\in\W^\alpha$ for all $t\geq0$ and $(\z,\y)$ satisfies the boundary jitter property (Definition \ref{def:jitter}). Then for all $\psi\in\cts(\R^J)$, $\nabla_\psi\esm^\alpha(\x)$ exists, lies in $\dlr(\R^J)$ and $\dm_\z^\alpha[\psi]$ is equal to the right-continuous regularization of $\nabla_\psi\esm^\alpha(\x)$; that is, $\dm_\z^\alpha[\psi](t)=\nabla_\psi\esm^\alpha(\x)(t+)$ for all $t\geq0$. In addition, $\nabla_\psi\esm^\alpha(\x)(\cdot)$ is continuous at times $t>0$ for which $\z(t)\in G^\circ\cup\U$.  
\end{prop}

\subsection{Proof of Theorem \ref{thm:pathwise}}\label{subs:pfpathwise}

We first show that whenever the directional derivative of the ESM evaluated at almost every sample path of $\X^{\alpha,x}$ exists, pathwise derivatives of reflected diffusions a.s.\ exist and can be characterized in terms of the directional derivative of the ESM. This result may be useful in cases when the boundary jitter property does not hold {(e.g., when the diffusion coefficient is degenerate)} but the directional derivative of the ESM still {exists}  ({e.g., for the class considered in  \cite{ManRam10}}). Recall the definition of $\partial_{\beta,y}\Z^{\alpha,x}$ given in \eqref{eq:nablaZ}.

\begin{prop}\label{prop:psivepsi}
Suppose Assumption \ref{ass:Holder} holds. Let $\alpha\in U$, $x\in G$ and suppose a.s.\ $\nabla_\psi\esm(\X^{\alpha,x})$ exists for all $\psi\in\cts(\R^J)$ and takes values in $\dlim(\R^J)$. Then for each $(\beta,y)\in\R^M\times G_x$, a.s.\ $\partial_{\beta,y}\Z^{\alpha,x}$ exists and is characterized as the unique $\{\F_t\}$-adapted process that satisfies $\partial_{\beta,y}\Z^{\alpha,x}=\nabla_{\Psi(\beta,y)}\esm(\X^{\alpha,x})$, where $\Psi(\beta,y)$ satisfies, for all $t\geq0$,
\begin{align}\label{eq:PsibetayPhi}
	\Psi_t(\beta,y)&= y+\int_0^tb'(\alpha,\Z_s^{\alpha,x})[\beta,\partial_{\beta,y}\Z_s^{\alpha,x}]ds+\int_0^t\sigma'(\alpha,\Z_s^{\alpha,x})[\beta,\partial_{\beta,y}\Z_s^{\alpha,x}]dW_s\\ \notag
	&\qquad+R'(\alpha)[\beta]\L_t^{\alpha,x},
\end{align}
and $\{R'(\alpha)[\beta]\L_t^{\alpha,x},t\geq0\}$ is the process described in Remark \ref{rmk:DRL}.
\end{prop}

\begin{remark}
Since functions in $\dlim(\R^J)$ {are} Lebesgue measurable and $\partial_{\beta,y}\Z^{\alpha,x}$ is $\{\F_t\}$-adapted, the Lebesgue-Stieltjes and It\^o integrals in \eqref{eq:PsibetayPhi} are well defined. 
\end{remark}

The proof of Proposition \ref{prop:psivepsi} is given in Section \ref{sec:psivepsi}. 

\begin{proof}[Proof of Theorem \ref{thm:pathwise}]
By assumption, a.s.\ $\tau^{\alpha,x}=\infty$ and $(\Z^{\alpha,x},\Y^{\alpha,x})$ satisfies the boundary jitter property. Thus, by Proposition \ref{prop:smderivative}, a.s.\ for all $\psi\in\cts(\R^J)$, 
\begin{itemize}
	\item[(a)] $\nabla_\psi\esm^\alpha(\X^{\alpha,x})$ exists and lies in $\dlr(\R^J)$,
	\item[(b)] $\nabla_\psi\esm^\alpha(\X^{\alpha,x})(t+)=\dm_{\Z^{\alpha,x}}^\alpha[\psi](t)$ for all $t\geq0$.
	\item[(c)] {$\nabla_\psi\esm^\alpha(\X^{\alpha,x})(\cdot)$ is continuous at times $t>0$ for which $\Z_t^{\alpha,x}\in G^\circ\cup\U$.}
\end{itemize}
Therefore, by Proposition \ref{prop:psivepsi}, for each $(\beta,y)\in\R^M\times G_x$, a.s.\
\begin{itemize}
	\item[(d)] $\partial_{\beta,y}\Z^{\alpha,x}$ exists,
	\item[(e)] $\partial_{\beta,y}\Z^{\alpha,x}=\nabla_{\Psi(\beta,y)}\esm(\X^{\alpha,x})$, where $\Psi(\beta,y)$ is defined in \eqref{eq:Psibetay}. 
\end{itemize}
Consequently, by (e), (a) and (b), for each $(\beta,y)\in\R^M\times G_x$, a.s.\
\begin{itemize}
	\item[(f)] $\partial_{\beta,y}\Z^{\alpha,x}$ lies in $\dlr(\R^J)$ 
	\item[(g)] $\lim_{s\downarrow t}\partial_{\beta,y}\Z_s^{\alpha,x}=\dm_{\Z^{\alpha,x}}^\alpha[\Psi(\beta,y)](t)$ for all $t\geq0$.
\end{itemize}
Thus, by (c)--(f), parts (i) and (ii) of Theorem \ref{thm:pathwise} hold. We are left to prove that there exists a {pathwise unique} derivative process $\phib^{\alpha,x}$ {along $Z^{\alpha,x}$} and part (iii) holds.  

For each $(\beta,y)\in\R^M\times G_x$, set $\Xi(\beta,y)\doteq\dm_{\Z^{\alpha,x}}^\alpha[\Psi(\beta,y)]$ so by (g) a.s.\ $\Xi_t(\beta,y)$ is the right-continuous regularization of $\partial_{\beta,y}\Z^{\alpha,x}$. By the definition of the derivative problem in Definition \ref{def:dp}, \eqref{eq:PsibetayPhi} and the fact that $\Xi(\beta,y)$ is the right continuous regularization of $\partial_{\beta,y}\Z^{\alpha,x}$, $\Xi(\beta,y)$ a.s.\ satisfies, for all $t\geq0$,
\begin{align}\label{eq:Xibetay}
	\Xi_t(\beta,y)&=y+\int_0^tb'(\alpha,\Z_s^{\alpha,x})[\beta,\Xi_s(\beta,y)]ds+\int_0^t\sigma'(\alpha,\Z_s^{\alpha,x})[\beta,\Xi_s(\beta,y)]dW_s\\ \notag
	&\qquad+R'(\alpha)[\beta]\L_t^{\alpha,x}+\Pi_t(\beta,y),
\end{align}
where $\Pi(\beta,y)=\{\Pi_t(\beta,y),t\geq0\}$ is a $J$-dimensional \cadlag process a.s.\ satisfying $\Pi_0(\beta,y)\in\spaan[d(\alpha,x)]$ and for all $0\leq s\leq t<\infty$,
	\be\label{eq:Pibetay}\Pi_t(\beta,y)-\Pi_s(\beta,y)\in\spaan\lsb\cup_{u\in(s,t]}d(\alpha,\Z_u^{\alpha,x})\rsb.\ee
Define $\wt y\doteq y+\Pi_0(\beta,y)$ and $\wt\Pi_t(\beta,y)\doteq\Pi_t(\beta,y)-\Pi_0(\beta,y)$ for all $t\geq0$. Observe that $\wt y=\Xi_0(\beta,y)\in H_x$ and $\wt y-y=\Pi_0(\beta,y)\in\spaan[d(\alpha,x)]$. Thus, by the uniqueness of the derivative projection operator shown in Lemma \ref{lem:projx}, it holds that $\wt y=\proj_x^\alpha[y]$. Therefore, using \eqref{eq:Xibetay}, we see that $\Xi(\beta,y)$ satisfies, for all $t\geq0$,
\begin{align}\label{eq:Xi}
	\Xi_t(\beta,y)&=\proj_x^\alpha[y]+\int_0^tb'(\alpha,\Z_s^{\alpha,x})[\beta,\Xi_s(\beta,y)]ds+\int_0^t\sigma'(\alpha,\Z_s^{\alpha,x})[\beta,\Xi_s(\beta,y)]dW_s\\ \notag
	&\qquad+R'(\alpha)[\beta]\L_t^{\alpha,x}+\wt\Pi_t(\beta,y),
\end{align}
where, by the definition of $\wt\Pi(\beta,y)$ and \eqref{eq:Pibetay}, $\wt\Pi_t(\beta,y)$ a.s.\ satisfies $\wt\Pi_0(\beta,y)=0$ and for all $0\leq s\leq t<\infty$,
	\be\label{eq:Pi}\wt\Pi_t(\beta,y)-\wt\Pi_s(\beta,y)\in\spaan\lsb\cup_{u\in(s,t]}d(\alpha,\Z_u^{\alpha,x})\rsb.\ee
	
Let $\{(\beta_k,y_k)\}_{k=1,\dots,m}$ denote an orthonormal basis of $\R^M\times H_x$. Since the basis is a finite set, a.s.\ $\Xi(\beta_k,y_k)$ satisfies \eqref{eq:Xi} for each $k=1,\dots,m$. Define $\phib^{\alpha,x}[\beta_k,y_k]\doteq\Xi(\beta_k,y_k)$ for each $k=1,\dots,m$, and linearly extend the definition of $\phib^{\alpha,x}[\cdot,\cdot]$ to all of $\R^M\times H_x$. Due to the linearity of \eqref{eq:Xi} and the fact $\proj_x^\alpha[y_k]=y_k$ for $k=1,\dots,m$, it follows that a.s.\ $\phib^{\alpha,x}[\beta,y]$ satisfies \eqref{eq:de} for all $(\beta,y)\in\R^M\times H_x$. Thus, $\phib^{\alpha,x}$ is a derivative process along $\Z^{\alpha,x}$, which is pathwise unique by Theorem \ref{thm:dpunique}. Moreover, it follows from \eqref{eq:Xi} that for any $(\beta,y)\in\R^M\times H_x$, a.s. $\Xi(\beta,y)=\phib^{\alpha,x}[\beta,\proj_x^\alpha[y]]$. This proves the remaining part (iii) of Theorem \ref{thm:pathwise}.
\end{proof}

\subsection{Proof of Proposition \ref{prop:psivepsi}}\label{sec:psivepsi}

Given $\alpha\in U$, $x\in G$, $\beta\in\R^M$ and $y\in G_x$, let $\ve_0(\alpha,x,\beta,y)>0$ be such that \eqref{eq:ve0} holds. Recall the definition of $\partial_{\beta,y}^\ve\Z^{\alpha,x}$ given in \eqref{eq:nablaveZ}, note that 
	\be\label{eq:Zve}\Z^{\alpha+\ve\beta,x+\ve y}=\Z^{\alpha,x}+\ve\partial_{\beta,y}^\ve\Z^{\alpha,x}\ee
and define the $J$-dimensional continuous process $\Psi^\ve(\beta,y)=\{\Psi_t^\ve(\beta,y),t\geq0\}$, for $t\geq0$, by
\begin{align}\label{eq:Psive}
	\Psi_t^\ve(\beta,y)&\doteq y+\int_0^t\frac{b(\alpha+\ve\beta,\Z_s^{\alpha,x}+\ve\partial_{\beta,y}^{\ve}\Z_s^{\alpha,x})-b(\alpha,\Z_s^{\alpha,x})}{\ve}ds\\ \notag
	&\qquad+\int_0^t\frac{\sigma(\alpha+\ve\beta,\Z_s^{\alpha,x}+\ve\partial_{\beta,y}^\ve\Z_s^{\alpha,x})-\sigma(\alpha,\Z_s^{\alpha,x})}{\ve}d\bm_s\\ \notag
	&\qquad+\frac{R(\alpha+\ve\beta)-R(\alpha)}{\ve}\L_t^{\alpha+\ve\beta,x+\ve y},
\end{align}
where, analogous to Remark \ref{rmk:DRL}, the last term is interpreted as follows (recall that Assumption \ref{ass:DRL} holds):
\begin{itemize}
	\item[1.] If Condition \ref{cond:independent} holds, then $\L^{\alpha+\ve\beta,x+\ve y}=\{\L_t^{\alpha+\ve\beta,x+\ve y},t\geq0\}$ is the $N$-dimensional process described in Remark \ref{rmk:local}.
	\item[2.] Otherwise, $R(\alpha)$ is constant in $\alpha\in U$ and we interpret the process to be identically zero (even if the process $\L^{\alpha+\ve\beta,x+\ve y}$ is not well defined).
\end{itemize}

\begin{lem}\label{lem:Zveesm}
Given $\alpha\in U$, $x\in G$, $\beta\in\R^M$ and $y\in G_x$, let $\ve_0(\alpha,x,\beta,y)>0$ be as in \eqref{eq:ve0}. For $0<\ve<\ve_0(\alpha,x,\beta,y)$, define $\Psi^\ve(\beta,y)$ as in \eqref{eq:Psive}. Then a.s.\
	\be\label{eq:ZveESM}\Z^{\alpha+\ve\beta,x+\ve y}=\esm^\alpha(\X^{\alpha,x}+\ve\Psi^\ve(\beta,y)).\ee
Consequently, a.s.\
	\be\label{eq:nablaveZbound}\partial_{\beta,y}^\ve\Z^{\alpha,x}=\frac{\esm^\alpha(\X^{\alpha,x}+\ve\Psi^\ve(\beta,y))-\esm^\alpha(\X^{\alpha,x})}{\ve}.\ee
Furthermore, if Condition \ref{cond:independent} holds and $\L^{\alpha+\ve\beta,x+\ve y}$ is the process defined in Remark \ref{rmk:local}, then a.s.\
	\be\label{eq:Lvebound}\norm{\L^{\alpha+\ve\beta,x+\ve y}}_t\leq\lip_\ell(\alpha)\lip_{\esm}(\alpha)\lb\norm{\X^{\alpha,x}-x}_t+\ve\norm{\Psi^\ve(\beta,y)}_t\rb.\ee
\end{lem}

\begin{proof}
Suppose \eqref{eq:ZveESM} holds. Then \eqref{eq:nablaveZbound} follows from \eqref{eq:nablaveZ}, \eqref{eq:ZveESM} and the fact that a.s.\ $\Z^{\alpha,x}=\esm(\X^{\alpha,x})$ by Remark \ref{rmk:X}. We now establish \eqref{eq:ZveESM}. According to Remark \ref{rmk:X}, 
	\be\label{eq:ZveesmXve}\Z^{\alpha+\ve\beta,x+\ve y}=\esm^{\alpha+\ve\beta}(\X^{\alpha+\ve\beta,x+\ve y}),\ee
and by \eqref{eq:X} and \eqref{eq:Psive},
	\be\label{eq:XveX}\X^{\alpha+\ve\beta,x+\ve y}=\X^{\alpha,x}+\ve\Psi^\ve(\beta,y)+\lb R(\alpha)- R(\alpha+\ve\beta)\rb\L^{\alpha+\ve\beta,x+\ve y},\ee
where the final term is taken to be zero if Condition \ref{cond:independent} does not hold. 

Suppose Condition \ref{cond:independent} holds. Then by Lemma \ref{lem:push}, \eqref{eq:ZveESM} holds. In addition, it follows from \eqref{eq:pushL} and the Lipschitz continuity of the ESM $\esm^\alpha$ that a.s., for all $t\geq0$,
	$$\norm{\L^{\alpha+\ve\beta,x+\ve y}}_t\leq\lip_\ell(\alpha)\norm{R(\alpha)\L^{\alpha+\ve\beta,x+\ve y}}_t\leq\lip_\ell(\alpha)\lip_{\esm}(\alpha)\lb\norm{\X^\alpha-x}_t+\ve\norm{\Psi^\ve(\beta,y)}_t\rb,$$
so \eqref{eq:Lvebound} holds. On the other hand, suppose Condition \ref{cond:independent} does not hold. Then Assumption \ref{ass:DRL} implies the directions of reflection are constant in $\alpha\in U$, so $\esm^{\alpha+\ve\beta}=\esm^\alpha$ and by convention, the final term on the right-hand side of \eqref{eq:XveX} is identically zero. Thus, \eqref{eq:ZveESM} follows from \eqref{eq:ZveesmXve} and \eqref{eq:XveX}.
\end{proof}

Given $\alpha\in U$, $x\in G$, $\beta\in\R^M$ and $y\in G_x$, let $\ve_0(\alpha,x,\beta,y)$ be as in \eqref{eq:ve0}. Set
	\be\label{eq:lipalphaxbetay}\lip(\alpha,x,\beta,y)\doteq\sup\lcb\lip_{\esm}(\alpha),\lip_{b,\sigma},\lip_R,\lip_\ell(\alpha+\ve\beta):0\leq\ve\leq\ve_0(\alpha,x,\beta,y)\rcb,\ee
and
	\be\label{eq:veast}\ve_\ast(\alpha,x,\beta,y)\doteq\min\lcb\ve_0(\alpha,x,\beta,y),\frac{1}{4(\lip(\alpha,x,\beta,y))^3|\beta|}\rcb>0.\ee
Since $\lip_\ell(\cdot)$ is bounded on compact subsets of $U$ by Lemma \ref{lem:push}, it follows that $\lip(\alpha,x,\beta,y)<\infty$.

\begin{lem}\label{lem:Xmomentbound}
Let $V\subset U$ and $K\subset G$ be compact subsets. Then for all $p\geq2$ and $t<\infty$, 
	\be\label{eq:uniformXmoment}\sup\{\E\lsb\norm{\X^{\alpha,x}-x}_t^p\rsb:\alpha\in V,x\in K\}<\infty.\ee
\end{lem}

\begin{proof}
The fact that $\E\lsb\norm{\X^{\alpha,x}-x}_t^p\rsb<\infty$ for fixed $\alpha\in U$ and $x\in G$ follows from a standard argument using \eqref{eq:X}, H\"older's inequality, the BDG inequalities, Tonelli's theorem, Assumption \ref{ass:drift}, the facts that $\Z^{\alpha,x}=\esm^\alpha(\X^{\alpha,x})$ and $\z=\esm^\alpha(\x)$ where $\z(\cdot)=\x(\cdot)\equiv x$, the Lipschitz continuity of the ESM $\esm^\alpha$ and Gronwall's inequality. The uniform bound \eqref{eq:uniformXmoment} then follows from \eqref{eq:XpKolmogorov}. 
\end{proof}

\begin{lem}\label{lem:psivemoment}
Let $\alpha\in U$, $x\in G$, $\beta\in\R^M$ and $y\in G_x$. Then for all $p\geq2$ and $t<\infty$,
	\be\label{eq:Psivebound}\sup\lcb\E\lsb\norm{\Psi^\ve(\beta,y)}_t^p\rsb:0<\ve<\ve_\ast(\alpha,x,\beta,y)\rcb<\infty.\ee
\end{lem}

\begin{proof}
For brevity, we set $\lip\doteq\lip(\alpha,x,\beta,y)$ and $\ve_\ast\doteq\ve_\ast(\alpha,x,\beta,y)$. Let $p\geq2$. Choose $0<\ve<\ve_\ast$ so that \eqref{eq:ve0} holds. By \eqref{eq:Psive} and \eqref{eq:X}, for all $t\geq0$,
\begin{align}\label{eq:EnormPsive}
	|\Psi_t^\ve(\beta,y)|^p&\leq 2^{p-1}|\X_t^{\alpha+\ve\beta,x+\ve y}-\X_t^{\alpha,x}|^p+2^{p-1}\left|\frac{R(\alpha+\ve\beta)-R(\alpha)}{\ve}\right|^p|\L_t^{\alpha+\ve\beta,x+\ve y}|^p.
\end{align}
According to Assumption \ref{ass:DRL}, the second term on the right-hand side of \eqref{eq:EnormPsive} is equal to zero if Condition \ref{cond:independent} does not hold. On the other hand, if Condition \ref{cond:independent} does hold, then \eqref{eq:Lvebound} holds. By \eqref{eq:Lvebound} and \eqref{eq:XpKolmogorov} of Lemma \ref{lem:ZalphaxZbetaybounds}, we have
\begin{align}\label{eq:Epsivep2}
	\E\lsb\norm{\Psi^\ve(\beta,y)}_t^p\rsb&\leq2^{p-1}\wt C^\dagger|y|^p+2^{p-1}\wt C^\ddagger|\beta|^p+4^{p-1}\lip^{3p}|\beta|^p\E\lsb\norm{\X^{\alpha,x}}_t^p\rsb\\ \notag
	&\qquad+\ve^p4^{p-1}\lip^{3p}|\beta|^p\E\lsb\norm{\Psi^\ve(\beta,y)}_t^p\rsb.
\end{align}
Rearranging, we obtain, for all $t\geq0$,
\begin{align*}
	\E\lsb\norm{\Psi^\ve(\beta,y)}_t^p\rsb&\leq\frac{2^{p-1}\wt C^\dagger|y|^p+2^{p-1}\wt C^\ddagger|\beta|^p+4^{p-1}\lip^{3p}|\beta|^p\E\lsb\norm{\X^{\alpha,x}}_t^p\rsb}{1-\ve_\ast^p4^{p-1}\lip^{3p}|\beta|^p},
\end{align*}
where \eqref{eq:veast} ensures $\ve_\ast^p4^{p-1}\lip^{3p}|\beta|^p<1$. Since the right-hand side of the last display does not depend on $0<\ve<\ve_\ast$, this completes the proof.
\end{proof}

\begin{lem}\label{lem:PhiPsi}
Suppose Assumption \ref{ass:Holder} holds. Let $\alpha\in U$, $x\in G$ and suppose a.s.\ $\nabla_\psi\esm^\alpha(\X^{\alpha,x})$ exists for all $\psi\in\cts(\R^J)$ and takes values in $\dlim(\R^J)$. Then for all $(\beta,y)\in\R^M\times G_x$, there exists a unique $J$-dimensional $\{\F_t\}$-adapted process $\Phi(\beta,y)=\{\Phi_t(\beta,y),t\geq0\}$ such that a.s.\ $\Phi(\beta,y)$ takes values in $\dlim(\R^J)$ and satisfies $\Phi(\beta,y)=\nabla_{\Psi(\beta,y)}\esm^\alpha(\X^{\alpha,x})$, where $\Psi(\beta,y)$ is a $J$-dimensional continuous $\{\F_t\}$-adapted process that satisfies, for all $t\geq0$,
\begin{align}\label{eq:Psibetay}
	\Psi_t(\beta,y)&= y+\int_0^tb'(\alpha,\Z_s^{\alpha,x})[\beta,\Phi_s(\beta,y)]ds+\int_0^t\sigma'(\alpha,\Z_s^{\alpha,x})[\beta,\Phi_s(\beta,y)]dW_s\\ \notag
	&\qquad+R'(\alpha)[\beta]\L_t^{\alpha,x},
\end{align}
where $\{R'(\alpha)[\beta]\L_t^{\alpha,x},t\geq0\}$ is the process defined in Remark \ref{rmk:DRL}. Moreover, for all $(\beta,y)\in\R^M\times G_x$, $p\geq2$ and $t<\infty$, 
	\be\label{eq:Psimomentbound}\E\lsb\norm{\Psi(\beta,y)}_t^p\rsb<\infty\ee 
and
	\be\label{eq:PsiveL2}\lim_{\ve\downarrow0}\E\lsb\norm{\Psi^\ve(\beta,y)-\Psi(\beta,y)}_t^2\rsb=0,\ee
where $\Psi^\ve(\beta,y)$ is defined as in \eqref{eq:Psive}.
\end{lem}

\begin{remark}
Since functions in $\dlim(\R^J)$ are Lebesgue measurable and $\Phi(\beta,y)$ is $\{\F_t\}$-adapted, both the Lebesgue-Stieltjes and It\^o integrals in \eqref{eq:Psibetay} are well defined. 
\end{remark}

\begin{proof}
Let $(\beta,y)\in\R^M\times G_x$ and set $\lip\doteq\max\{\lip(\alpha,x,\beta,y),\lip'\}$, where $\lip(\alpha,x,\beta,y)$ and $\lip'$ are the constants in 
\eqref{eq:lipalphaxbetay} and Assumption \ref{ass:Holder}, respectively. We first show uniqueness. Suppose there are two such process $\Phi(\beta,y)$ and $\wt\Phi(\beta,y)$. Using \eqref{eq:Psibetay} and standard estimates involving the Cauchy-Schwartz inequality, the BDG inequalities, Tonelli's theorem and Assumption \ref{ass:drift}, we have, for $t\geq0$,
\begin{align*}
	\E\lsb\norm{\Psi(\beta,y)-\wt\Psi(\beta,y)}_t^2\rsb&\leq2(t+C_2)\lip^2\int_0^t\E\lsb\norm{\Phi(\beta,y)-\wt\Phi(\beta,y)}_s^2\rsb ds
\end{align*}
The Lipschitz continuity of the function $\psi\mapsto\nabla_\psi\esm^\alpha(\X^{\alpha,x})$ shown in Proposition \ref{prop:nablaesmlip} along with an application of Gronwall's inequality implies that a.s.\ $(\Phi(\beta,y),\Phi(\beta,y))=(\wt\Phi(\beta,y),\wt\Phi(\beta,y))$.

The proof of existence of the process $\Phi(\beta,y)$ follows a standard Picard iteration argument. Set $\Phi^0\doteq0$ and recursively define, for $t\geq0$,
\begin{align}\label{eq:Psik}
	\Psi_t^k&\doteq y+\int_0^tb'(\alpha,\Z_s^{\alpha,x})[\beta,\Phi_s^{k-1}]ds+\int_0^t\sigma'(\alpha,\Z_s^{\alpha,x})[\beta,\Phi_s^{k-1}]dW_s+R'(\alpha)[\beta]\L_t^{\alpha,x},
\end{align}
and set $\Phi^k\doteq\nabla_{\Psi^k}\esm(\X^{\alpha,x})$,  where the integrals are well defined because $\Phi^k$ takes values in $\dlim(\R^J)$ and is $\{\F_t\}$-adapted because $\Psi^k$ is $\{\F_t\}$-adapted and the function $\psi\mapsto\nabla_\psi\esm^\alpha(\X^{\alpha,x})$ is Lipschitz continuous. Using \eqref{eq:Psik} and standard estimates as above, we obtain,
\begin{align}
\label{Psi1bd}
	\E\lsb\norm{\Psi^1}_t^2\rsb&\leq4|y|^2+4(t+C_2)\lip^2|\beta|^2+4\E\lsb\norm{R'(\alpha)[\beta]\L^{\alpha,x}}_t^2\rsb,\qquad t\geq0,
\end{align}
and for each $k\in\N$, again using standard estimates along with the Lipschitz continuity the function $\psi\mapsto\nabla_\psi\esm^\alpha(\X^{\alpha,x})$,
\begin{align*}
	\E\lsb\norm{\Psi^{k+1}-\Psi^k}_t^2\rsb&\leq 2(t+C_2)\lip^4\int_0^t\E\lsb\norm{\Psi^k-\Psi^{k-1}}_s^2\rsb ds,\qquad t\geq0.
\end{align*}
Iterating the last display yields, for each $k\in\N$,
\begin{align}\label{eq:Psik1Psik}
	\E\lsb\norm{\Psi^{k+1}-\Psi^k}_t^2\rsb&\leq\frac{(2(t+C_2)\lip^4)^k}{k!}\int_0^t\E\lsb\norm{\Psi^1}_s^2\rsb ds.
\end{align}
The first two terms on the right-hand side of \eqref{Psi1bd} are clearly finite. If Condition \ref{cond:independent} does not hold, then by Assumption \ref{ass:DRL}, the last term is defined to be zero (see Remark \ref{rmk:DRL}). On the other hand, if Condition \ref{cond:independent} holds, then let $\L^{\alpha,x}$ be the $N$-dimensional process introduced in Remark \ref{rmk:local}. Then by the fact that $(\z,\y)\equiv(x,0)$ is the solution to the ESP $\{(d_i(\alpha),n_i,c_i)\}$ for $\x\equiv x$, \eqref{eq:pushL}, Proposition \ref{prop:esmlip}, \eqref{eq:lipalphaxbetay} and Lemma \ref{lem:Xmomentbound}, for $p\geq1$,
	\be\label{eq:RLbound}\E\lsb\norm{R'(\alpha)[\beta]\L^{\alpha,x}}_t^p\rsb\leq\lip^{3p}|\beta|^p\E\lsb\norm{\X^{\alpha,x}-x}_t^p\rsb<\infty.\ee
Thus, $\E\lsb\norm{\Psi^1}_t^2\rsb<\infty$, which along with \eqref{eq:Psik1Psik}, implies $\lcb\E\lsb\norm{\Psi^{k+1}-\Psi^k}_t^2\rsb\rcb_{k\in\N}$ is a Cauchy sequence. Then by a standard argument using Chebyshev's inequality and the Borel-Cantelli lemma, there must exist a continuous process $\Psi$ such that a.s.\ $\Psi^k$ converges to $\Psi$ in $\cts(\R^J)$ as $k\to\infty$. Due to the relation $\Phi^k\doteq\nabla_{\Psi^k}\esm(\X^{\alpha,x})$ and the Lipschitz continuity of $\psi\mapsto\nabla_\psi\esm(\X^{\alpha,x})$ shown in Proposition \ref{prop:nablaesmlip}, a.s.\ $\Phi^k$ converges to $\Phi$ uniformly on compact time intervals as $k\to\infty$. Hence, by \eqref{eq:Psik} and the continuity of $\psi\mapsto\nabla_\psi\esm(\X^{\alpha,x})$, we see that $\Psi$ satisfies \eqref{eq:Psibetay} and $\Phi=\nabla_\Psi\esm(\X^{\alpha,x})$.

Next, we show \eqref{eq:Psimomentbound} holds. By \eqref{eq:Psibetay}, H\"older's inequality, the BDG inequalities, Tonelli's theorem, Assumption \ref{ass:drift}, the fact that $\Phi(\beta,y)=\nabla_{\Psi(\beta,y)}\esm(\X^{\alpha,x})$, the Lipschitz continuity of $\psi\mapsto\nabla_\psi\esm(\X^{\alpha,x})$ shown in Proposition \ref{prop:nablaesmlip}, \eqref{eq:lipalphaxbetay} and \eqref{eq:RLbound}, for all $t\geq0$,
\begin{align*}
	\E\lsb\norm{\Psi(\beta,y)}_t^p\rsb&\leq4^{p-1}|y|^p+4^{p-1}\E\lsb\sup_{0\leq s\leq t}\left|\int_0^sb'(\alpha,\Z_u^{\alpha,x})[\beta,\Phi_u(\beta,y)]du\right|^p\rsb\\
	&\qquad+4^{p-1}\E\lsb\sup_{0\leq s\leq t}\left|\int_0^s\sigma'(\alpha,\Z_u^{\alpha,x})[\beta,\Phi_u(\beta,y)]d\bm_u\right|^p\rsb\\
	&\qquad+4^{p-1}\lip^{3p}|\beta|^p\E\lsb\norm{\X^{\alpha,x}-x}_t^p\rsb\\
	&\leq4^{p-1}|y|^p+8^{p-1}(t^{p-1}+C_p)|\beta|^pt+4^{p-1}\lip^{3p}|\beta|^p\E\lsb\norm{\X^{\alpha,x}-x}_t^p\rsb\\
	&\qquad+8^{p-1}(t^{p-1}+C_p)\lip^{2p}\int_0^t\E\lsb\norm{\Psi(\beta,y)}_s^p\rsb ds.
\end{align*}
Then by Gronwall's inequality and \eqref{eq:RLbound}, \eqref{eq:Psimomentbound} holds.

Lastly, we prove \eqref{eq:PsiveL2}. Recall that by Assumption \ref{ass:Holder}, $b'$, $\sigma'$ and $R'$ are $\gamma$-H\"older continuous. Let $0<\ve<\ve_\ast(\alpha,x,\beta,y)$. By \eqref{eq:Psive} and \eqref{eq:Psibetay}, for $t\geq0$, 
\begin{align}\label{eq:psitvepsit}
	\E\lsb\norm{\Psi^\ve(\beta,y)-\Psi(\beta,y)}_t^2\rsb&\leq3\E\lsb\sup_{s\in[0,t]}\left|\int_0^s\Delta_u^{(b)}du\right|^2\rsb+3\E\lsb\sup_{s\in[0,t]}\left|\int_0^s\Delta_u^{(\sigma)}dW_u\right|^2\rsb\\ \notag
	&\qquad+3\E\lsb\norm{\Delta^{(R)}}_t^2\rsb,
\end{align}
where, for $s\geq0$,
\begin{align}\label{eq:Deltab}
	\Delta_s^{(b)}&\doteq\frac{b(\alpha+\ve\beta,\Z_s^{\alpha,x}+\ve\partial_{\beta,y}^{\ve}\Z_s^{\alpha,x})-b(\alpha,\Z_s^{\alpha,x})}{\ve}-b'(\alpha,\Z_s^{\alpha,x})[\beta,\Phi_s(\beta,y)],\\ \label{eq:Deltasigma}
	\Delta_s^{(\sigma)}&\doteq\frac{\sigma(\alpha+\ve\beta,\Z_s^{\alpha,x}+\ve\partial_{\beta,y}^\ve\Z_s^{\alpha,x})-\sigma(\alpha,\Z_s^{\alpha,x})}{\ve}-\sigma'(\alpha,\Z_s^{\alpha,x})[\beta,\Phi_s(\beta,y)],\\ \label{eq:DeltaR}
	\Delta_s^{(R)}&\doteq\frac{R(\alpha+\ve\beta)-R(\alpha)}{\ve}\L_s^{\alpha+\ve\beta,x+\ve y}-R'(\alpha)[\beta]\L_s^{\alpha,x}.
\end{align}
For $f=b,\sigma$, by \eqref{eq:Deltab} and \eqref{eq:Deltasigma},
\begin{align*}
	\Delta_s^{(f)}&=\frac{f(\alpha+\ve\beta,\Z_s^{\alpha,x}+\ve\partial_{\beta,y}^{\ve}\Z_s^{\alpha,x})-f(\alpha+\ve\beta,\Z_s^{\alpha,x}+\ve\Phi_s(\beta,y))}{\ve}\\
	&\qquad+\int_0^1\lcb f'(\alpha+v\ve\beta,\Z_s^{\alpha,x}+v\ve\Phi_s(\beta,y))-f'(\alpha,\Z_s^{\alpha,x})\rcb[\beta,\Phi_s(\beta,y)]dv.
\end{align*}
By the last display, the Lipschitz continuity of $b$ and $\sigma$ implied by Assumption \ref{ass:drift}, Jensen's inequality, the $\gamma$-H\"older continuity of $b'$ and $\sigma'$, the fact that $\Phi(\beta,y)=\nabla_{\Psi(\beta,y)}\esm^\alpha(\X^{\alpha,x})$, the Lipschitz continuity of $\psi\mapsto\nabla_\psi\esm^\alpha(\X^{\alpha,x})$ stated in Proposition \ref{prop:nablaesmlip}, and \eqref{eq:lipalphaxbetay},
\begin{align}\label{eq:Deltatvarphi2}
	|\Delta_s^{(f)}|^2&\leq2\lip^2|\partial_{\beta,y}^\ve\Z_s^{\alpha,x}-\Phi_s(\beta,y)|^2+2\lip^2\ve^{2\gamma}|\beta|^{2+2\gamma}+2\lip^{4+2\gamma}\ve^{2\gamma}\norm{\Psi(\beta,y)}_s^{2+2\gamma}.
\end{align}
By the Cauchy-Schwarz inequality, the BDG inequalities, Tonelli's theorem and \eqref{eq:Deltatvarphi2},
\begin{align}\label{eq:Deltabt2}
	&\E\lsb\sup_{s\in[0,t]}\left|\int_0^s\Delta_u^{(b)}du\right|^2\rsb+\E\lsb\sup_{s\in[0,t]}\left|\int_0^s\Delta_u^{(\sigma)}dW_u\right|^2\rsb\\ \notag
	&\qquad\leq t\int_0^t\E\lsb|\Delta_s^{(b)}|^2\rsb ds+C_2\int_0^t\E\lsb|\Delta_s^{(\sigma)}|^2\rsb ds\\ \notag
	&\qquad\leq 2(t+C_2)\lip^2\int_0^t\E\lsb\left|\partial_{\beta,y}^\ve\Z_s^{\alpha,x}-\Phi_s(\beta,y)\right|^2\rsb ds\\ \notag
	&\qquad\qquad+2t(t+C_2)\lip^2\ve^{2\gamma}|\beta|^{2+2\gamma}+2t(t+C_2)\lip^{4+2\gamma}\ve^{2\gamma}\E\lsb\norm{\Psi(\beta,y)}_t^{2+2\gamma}\rsb.
\end{align}
By \eqref{eq:DeltaR},
\begin{align*}
	\Delta_s^{(R)}&=\frac{R(\alpha+\ve\beta)-R(\alpha)}{\ve}(\L_s^{\alpha+\ve\beta,x+\ve y}-\L_s^{\alpha,x})+\int_0^1\lcb R'(\alpha+v\ve\beta)-R'(\alpha)\rcb[\beta]\L_s^{\alpha,x}dv.
\end{align*}
By \eqref{eq:pushL}, \eqref{eq:ZveESM}, the Lipschitz continuity of $\esm^\alpha$, the $\gamma$-H\"older continuity of $R'$ and \eqref{eq:lipalphaxbetay},
\begin{align*}
	|\Delta_s^{(R)}|^2&\leq 2\left|\frac{R(\alpha+\ve\beta)-R(\alpha)}{\ve}\right|^2|\L_s^{\alpha+\ve\beta,x+\ve y}-\L_s^{\alpha,x}|^2\\
	&\qquad+2\left|\int_0^1\lcb R'(\alpha+v\ve\beta)-R'(\alpha)\rcb[\beta]dv\right|^2|\L_s^{\alpha,x}|^2\\
	&\leq2\lip^4|\beta|^2\ve^2\norm{\Psi^\ve(\beta,y)}_s^2+2\lip^4\ve^{2\gamma}|\beta|^{2+2\gamma}\norm{\X^{\alpha,x}-x}_s^2.
\end{align*}
Taking the expectation of the supremum over $s\in[0,t]$, we obtain 
\begin{align}\label{eq:DeltaRt2}
	\E\lsb\norm{\Delta^{(R)}}_t^2\rsb&\leq2\lip^4|\beta|^2\ve^2\E\lsb\norm{\Psi^\ve(\beta,y)}_t^2\rsb+2\lip^4\ve^{2\gamma}|\beta|^{2+2\gamma}\E\lsb\norm{\X^{\alpha,x}-x}_t^2\rsb.
\end{align}
Substituting \eqref{eq:Deltabt2} and \eqref{eq:DeltaRt2} into \eqref{eq:psitvepsit} yields
\begin{align}\label{eq:psitvepsit2}
	\E\lsb\norm{\Psi^\ve(\beta,y)-\Psi(\beta,y)}_t^2\rsb&\leq6(t+C_2)\lip^2\int_0^t\E\lsb\left|\partial_{\beta,y}^\ve\Z_s^{\alpha,x}-\Phi_s(\beta,y)\right|^2\rsb ds\\ \notag
	&\qquad+6\ve^{2\gamma}t(t+C_2)\lip^2|\beta|^{2+2\gamma}\\ \notag
	&\qquad+6\ve^{2\gamma}t(t+C_2)\lip^{4+2\gamma}\E\lsb\norm{\Psi_t(\beta,y)}_t^{2+2\gamma}\rsb\\ \notag
	&\qquad+6\ve^2\lip^4\beta^2\E\lsb\norm{\Psi^\ve(\beta,y)}_t^2\rsb\\ \notag
	&\qquad+6\ve^{2\gamma}\lip^4|\beta|^{2+2\gamma}\E\lsb\norm{\X^{\alpha,x}-x}_t^2\rsb.
\end{align}
By \eqref{eq:nablaveZbound}, the fact that
$\Phi(\beta,y)=\nabla_{\Psi(\beta,y)}\esm^\alpha(\X^{\alpha,x})$, the
triangle inequality, \eqref{eq:nablavepsiesm}, the Lipschitz
continuity of the ESM $\esm^\alpha$, for $s\geq0$,
\begin{align}\label{eq:nablaZphib}
	|\partial_{\beta,y}^\ve\Z_s^{\alpha,x}-\Phi_s(\beta,y)|^2&\leq2|\nabla_{\Psi^\ve(\beta,y)}^\ve\esm^\alpha(\X^{\alpha,x})(s)-\nabla_{\Psi(\beta,y)}^\ve\esm^\alpha(\X^{\alpha,x})(s)|^2\\ \notag
	&\qquad+2|\nabla_{\Psi(\beta,y)}^\ve\esm^\alpha(\X^{\alpha,x})(s)-\nabla_{\Psi(\beta,y)}\esm^\alpha(\X^{\alpha,x})(s)|^2\\ \notag
	&\leq2\lip^2\norm{\Psi^\ve(\beta,y)-\Psi(\beta,y)}_s^2 \\
    \notag 
	&\qquad+2|\nabla_{\Psi(\beta,y)}^\ve\esm^\alpha(\X^{\alpha,x})(s)-\nabla_{\Psi(\beta,y)}\esm^\alpha(\X^{\alpha,x})(s)|^2.
\end{align}
Therefore, upon substituting \eqref{eq:nablaZphib} into \eqref{eq:psitvepsit2}, we obtain for all $t\geq0$,
\begin{align*} 
\E\lsb\norm{\Psi^\ve(\beta,y)-\Psi(\beta,y)}_t^2\rsb & \leq C_t (\ve)+12(t+C_2)\lip^4\int_0^t\E\lsb\norm{\Psi^\ve(\beta,y)-\Psi(\beta,y)}_s^2\rsb ds, 
\end{align*}
where 
\begin{align}\label{eq:Ctve}
C_t(\ve)&\doteq12(t+C_2)\lip^2\int_0^t\E\lsb|\nabla_{\Psi(\beta,y)}^\ve\esm^\alpha(\X^{\alpha,x})(s)-\nabla_{\Psi(\beta,y)}\esm^\alpha(\X^{\alpha,x})(s)|^2\rsb
{ds}\\ \notag
	&\qquad+6\ve^{2\gamma}t(t+C_2)\lip^2|\beta|^{2+2\gamma}+6\ve^{2\gamma}t(t+C_2)\lip^{4+2\gamma}\E\lsb\norm{\Psi_t(\beta,y)}_t^{2+2\gamma}\rsb\\ \notag
	&\qquad+6\ve^2\lip^4|\beta|^2\E\lsb\norm{\Psi^\ve(\beta,y)}_t^2\rsb+6\ve^{2\gamma}\lip^4|\beta|^{2+2\gamma}\E\lsb\norm{\X^{\alpha,x}-x}_t^2\rsb.
\end{align}
Gronwall's inequality then implies
	$$\E\lsb\norm{\Psi^\ve(\beta,y)-\Psi(\beta,y)}_t^2\rsb\leq C_t(\ve)\exp\lb12(t+C_2)\lip^4t\rb.$$
Let $t<\infty$. Once we demonstrate that $C_t(\ve)\to0$ as $\ve\downarrow0$, the proof of the lemma will be complete. Due to \eqref{eq:Psimomentbound}, Lemma \ref{lem:psivemoment} and Lemma \ref{lem:Xmomentbound}, the last four terms on the right-hand side of \eqref{eq:Ctve} converge to zero as $\ve\downarrow0$. We are left to show that
	\be\label{eq:Enablavenablazero}\lim_{\ve\downarrow0}\int_0^t\E\lsb\left|\nabla_{\Psi(\beta,y)}^\ve\esm^\alpha(\X^{\alpha,x})(s)-\nabla_{\Psi(\beta,y)}\esm^\alpha(\X^{\alpha,x})(s)\right|^2\rsb ds=0.\ee
By \eqref{eq:nablavepsiesm}, the Lipschitz continuity of the ESM $\esm^\alpha$, Proposition \ref{prop:nablaesmlip} and the fact that $\nabla_\psi \bar{\Gamma} (f) \equiv 0$ when $\psi \equiv 0$,
\begin{align}\label{eq:EnablaPsi2bound}
	\left|\nabla_{\Psi(\beta,y)}^\ve\esm^\alpha(\X^{\alpha,x})(t)-\nabla_{\Psi(\beta,y)}\esm^\alpha(\X^{\alpha,x})(t)\right|^2&\leq4\lip^2\norm{\Psi(\beta,y)}_t^2.
\end{align}
Together with \eqref{eq:Psimomentbound}, \eqref{eq:nablapsiesm} and the dominated convergence theorem, this implies \eqref{eq:Enablavenablazero}. 
\end{proof}

\begin{proof}[Proof of Proposition \ref{prop:psivepsi}]
Let $(\beta,y)\in\R^M\times G_x$ and $(\Phi(\beta,y),\Psi(\beta,y))$ be as in Lemma \ref{lem:PhiPsi}. By \eqref{eq:nablaZ}, Lemma \ref{lem:Zveesm} and Proposition \ref{prop:esmpsive}, a.s.\ $\partial_{\beta,y}\Z^{\alpha,x}=\nabla_{\Psi(\beta,y)}\esm^\alpha(\X^{\alpha,x})$. Then according to Lemma \ref{lem:PhiPsi}, a.s.\ $\partial_{\beta,y}\Z^{\alpha,x}=\Phi(\beta,y)$. Thus, \eqref{eq:Psibetay} implies \eqref{eq:PsibetayPhi} holds.
\end{proof}

\appendix

\section{Proof of Lemma \ref{lem:push}}\label{apdx:push}

\begin{proof}[Proof of Lemma \ref{lem:push}]
By \cite[Definition 1.1]{Ramanan2006}, there exists a Lebesgue measurable function $\xi:[0,\infty)\mapsto \mathbb{S}^{J-1}$ such that $\xi(t)\in d(\alpha,\z(t))$ ($d|\y|$-almost everywhere) and
	\be\label{eq:yty0gamma}\y(t)=\int_{(0,t]}\xi(s)d|\y|(s).\ee 
By \eqref{eq:dx} and Condition \ref{cond:independent}, given $x\in G$, there is a unique continuous function $\zeta_{\allN(x)}:d(\alpha,x)\mapsto\R_+^N$ such that for all $y\in d(\alpha,x)$,
	\be\label{eq:zetax}y=\sum_{i\in\allN(x)}\zeta_{\allN(x)}^i(y)d_i(\alpha)\qquad\text{and}\qquad\zeta_{\allN(x)}^i(y)=0\text{ for all }i\not\in\allN(x).\ee
For the $d|g|$-almost every $t\geq0$ such that $\xi(t)\in d(\alpha,\z(t))$, define $\chi(t)\doteq\zeta_{\z(t)}(\xi(t))$. Since $\xi$ is Lebesgue measurable, $\allN(\cdot)$ is upper semicontinuous (Lemma \ref{lem:allNusc}), $\z$ is continuous and $\zeta_{\allN(x)}$ is continuous for each $x\in G$, it follows that $\chi$ is also Lebesgue measurable. For each $i\in\allN$, define
	\be\label{eq:Lizetai}\push^i(t)\doteq\int_{(0,t]}\chi^i(s)d|\y|(s)= \int_{(0,t]}1_{\{i\in\allN(\z(s))\}}\chi^i(s)d|\y|(s),\ee
where the second equality follows because $\chi^i(s)=\zeta_{h(s)}^i(\xi(s))=0$ if $i\not\in\allN(\z(s))$. Since $\chi$ takes values in $\R_+^N$, \eqref{eq:Lizetai} implies that for each $i\in\allN$, $\push^i$ is nondecreasing and \eqref{eq:dLi} holds. By \eqref{eq:yty0gamma}--\eqref{eq:Lizetai}, for all $t\in[0,\infty)$,
	$$\y(t)=\sum_{i\in\allN}\lb\int_{(0,t]}\chi^i(s)d|\y|(s)\rb d_i(\alpha)=R(\alpha)\push(t).$$
Together with the linear independence condition in Condition \ref{cond:independent}, this implies that $\push$ is uniquely defined and there exists a positive constant $\lip_\ell(\alpha)<\infty$ such that if, for $k=1,2$, $(\z_k,\y_k)$ is the solution to the ESP $\{(d_i(\alpha),n_i,c_i),i\in\allN\}$ for $\x_k\in\cts_G$ and $\push_k$ is as above, but with $\z_k,\y_k$ and $\push_k$ in place of $\z,\y$ and $\push$, then for all $t<\infty$, $\norm{\push_1-\push_2}_t\leq\lip_\ell(\alpha)\norm{\y_1-\y_2}_t$. The continuity of $R(\cdot)$ implies that $\lip(\cdot)$ can be chosen to be bounded on compact subsets of $U$. To prove the final statement of the lemma, let $\wt\alpha\in U$. By Definition \ref{def:esp}, the fact that $\y=R(\alpha)\ell$ and the definition of $\wt\x$, we have $\z=\x+R(\alpha)\ell=\wt\x+R(\wt\alpha)\ell$. Since $\x$ takes values in $G$ and for each $i\in\allN$, $\ell^i$ starts at zero, is nondecreasing and can only increase when $\z$ lies in face $F_i$, it follows that $(\z,R(\wt\alpha)\ell)$ is a solution the ESP $\{(d_i(\wt\alpha),n_i,c_i),i\in\allN\}$ for $\wt\x$.
\end{proof}

\section{Proof of Lemma \ref{lem:ZalphaxZbetaybounds}}\label{sec:continuousfield}

\begin{proof}[Proof of Lemma \ref{lem:ZalphaxZbetaybounds}]
Fix $p\geq2$, $t<\infty$ and compact subsets $V\subset U$ and $K\subset G$. Let $\alpha_0\in V$ and set $\lip\doteq\sup\{\lip_{b,\sigma},\lip_R,\lip_\ell(\alpha),\lip_{\esm}(\alpha_0):\alpha\in V\}$, where $\lip<\infty$ follows because $\lip_\ell(\cdot)$ is bounded on compact subsets of $U$ by Lemma \ref{lem:push}. Let $(\alpha,x),(\wt\alpha,\wt x)\in V\times K$ and define $\X^{\alpha,x}$ as in \eqref{eq:X} and $\X^{\wt\alpha,\wt x}$ as in \eqref{eq:X}, but with $\wt\alpha$ and $\wt x$ in place of $\alpha$ and $x$, respectively. Then by \eqref{eq:X}, H\"older's inequality, the BDG inequalities, the Lipschitz continuity of $b(\cdot,\cdot)$ and $\sigma(\cdot,\cdot)$ that follows from Assumption \ref{ass:drift} and Tonelli's theorem,
\begin{align}\label{eq:XalphaxXbetay}
	\E\lsb\norm{\X^{\alpha,x}-\X^{\wt\alpha,\wt x}}_t^p\rsb&\leq3^{p-1}|x-\wt x|^p\\ \notag
	&\qquad+3^{p-1}(t^{p-1}+C_p)\lip^p\int_0^t\E\lsb\sup_{0\leq u\leq s}|(\alpha-\wt\alpha,\Z_u^{\alpha,x}-\Z_u^{\wt\alpha,\wt x})|^p\rsb ds\\ \notag
	&\leq3^{p-1}|x-\wt x|^p+3^{p-1}(t^{p-1}+C_p)\lip^pt|\alpha-\wt\alpha|^p\\ \notag
	&\qquad+3^{p-1}(t^{p-1}+C_p)\lip^p\int_0^t\E\lsb\norm{\Z^{\alpha,x}-\Z^{\wt\alpha,\wt x}}_s^p\rsb ds.
\end{align}
We consider two cases.\newline\newline
\emph{Case 1:} Suppose Condition \ref{cond:independent} holds. By Lemma \ref{lem:push}, 
\begin{align*}
	\Z^{\alpha,x}&=\esm^{\alpha_0}(\X^{\alpha,x}+(R(\alpha)-R(\alpha_0))\L^{\alpha,x}),\\
	\Z^{\wt\alpha,\wt x}&=\esm^{\alpha_0}(\X^{\wt\alpha,\wt x}+(R(\wt\alpha)-R(\alpha_0))\L^{\wt\alpha,\wt x}).
\end{align*}
Therefore, by Proposition \ref{prop:esmlip}, \eqref{eq:pushL} and the fact that $(\z,\y)\equiv(x,0)$ is a solution to the ESP $\{(d_i(\alpha),n_i,c_i),i\in\allN\}$ for $\x\equiv x$,
\begin{align*}
	\norm{\Z^{\alpha,x}-\Z^{\wt\alpha,\wt x}}_t&\leq\lip\norm{\X^{\alpha,x}-\X^{\wt\alpha,\wt x}}_t+\lip|R(\alpha)-R(\wt\alpha)|\norm{\L^{\alpha,x}}_t+\lip|R(\wt\alpha)-R(\alpha_0)|\norm{\L^{\alpha,x}-\L^{\wt\alpha,\wt x}}_t\\
	&\leq\lip\norm{\X^{\alpha,x}-\X^{\wt\alpha,\wt x}}_t+\lip^4|\alpha-\wt\alpha|\norm{\X^{\alpha,x}-x}_t+\lip^{4}|\wt\alpha-\alpha_0|\norm{\X^{\alpha,x}-\X^{\wt\alpha,\wt x}}_t\\
	&\leq\lip(1+\lip^3|\wt\alpha-\alpha_0|)\norm{\X^{\alpha,x}-\X^{\wt\alpha,\wt x}}_t+\lip^4|\alpha-\wt\alpha|\norm{\X^{\alpha,x}-x}_t.
\end{align*}
\newline
\emph{Case 2:} Suppose Condition \ref{cond:independent} does not hold, so by Assumption \ref{ass:DRL}, $R(\alpha)$ is constant in $\alpha\in U$. Then according to Remark \ref{rmk:X}, $\Z^{\alpha,x}=\esm^{\alpha_0}(\X^{\alpha,x})$ and $\Z^{\wt\alpha,\wt x}=\esm^{\alpha_0}(\X^{\wt\alpha,\wt x})$, so by Proposition \ref{prop:esmlip}, for all $t\geq0$,
	$$\norm{\Z^{\alpha,x}-\Z^{\wt\alpha,\wt x}}_t\leq\lip\norm{\X^{\alpha,x}+\X^{\wt\alpha,\wt x}}_t.$$

In either case, by \eqref{eq:XalphaxXbetay},
\begin{align*}
	\E\lsb\norm{\X^{\alpha,x}-\X^{\wt\alpha,\wt x}}_t^p\rsb&\leq 3^{p-1}|x-\wt x|^p+3^{p-1}(t^{p-1}+C_p)\lip^pt|\alpha-\wt\alpha|^p\\
	&\qquad+6^{p-1}(t^{p-1}+C_p)\lip^{5p}|\alpha-\wt\alpha|^p\E\lsb\norm{\X^{\alpha,x}-x}_t^p\rsb\\ \notag
	&\qquad+6^{p-1}(t^{p-1}+C_p)\lip^{2p}(1+\lip^3|\wt\alpha-\alpha_0|)^p\int_0^t\E\lsb\norm{\X^{\alpha,x}-\X^{\wt\alpha,\wt x}}_s^p\rsb ds.
\end{align*}
An application of Gronwall's inequality yields \eqref{eq:XpKolmogorov} with 
\begin{align*}
	\wt C^\dagger\doteq&\sup\lcb3^{p-1}(t^{p-1}+C_p)\lip^p\lb t+2^{p-1}\lip^{4p}\E\lsb\norm{\X^{\alpha,x}-x}_t^p\rsb\rb:\wt\alpha\in V,x\in K\rcb\\
	&\qquad\times\sup\lcb\exp\lb6^{p-1}(t^{p-1}+C_p)\lip^{2p}(1+\lip^3|\wt\alpha-\alpha_0|)^pt\rb:\wt\alpha\in V\rcb,\\
	\wt C^\ddagger\doteq&\sup\lcb3^{p-1}\exp\lb6^{p-1}(t^{p-1}+C_p)\lip^{2p}(1+\lip^3|\wt\alpha-\alpha_0|)^pt\rb:\wt\alpha\in V\rcb.
\end{align*}
Here $|\wt\alpha-\alpha_0|$ is uniformly bounded over $\wt\alpha\in V$ since $V$ is compact, and $\E\lsb\norm{\X^{\alpha,x}-x}_t^p\rsb$ is uniformly bounded over $\alpha\in V$ and $x\in K$ by Lemma \ref{lem:Xmomentbound}. It then follows from \eqref{eq:XpKolmogorov} and the bounds shown in Cases 1 and 2 above that \eqref{eq:ZpKolmogorov} holds with 
\begin{align*}
	C^\dagger&\doteq2^{p-1}(\lip(1+\lip^3|\wt\alpha-\alpha_0|))^p\wt C^\dagger+2^{p-1}\lip^{4p}\E\lsb\norm{\X^{\alpha,x}-x}_t^p\rsb,\\
	C^\ddagger&\doteq2^{p-1}(\lip(1+\lip^3|\wt\alpha-\alpha_0|))^p\wt C^\ddagger.
\end{align*}
\end{proof}

\section{Proof of Lemma \ref{lem:ZkYkESPXk}}\label{apdx:ZkYkESPXk}

\begin{proof}[Proof of Lemma \ref{lem:ZkYkESPXk}]
Since $\{(d_i(\cdot),n_i,c_i),i\in\allN\}$ satisfies Assumption \ref{ass:setB}, for each $\alpha\in U$ there is a set $B^\alpha$ such that \eqref{eq:setB} holds for all $i\in\allN$, and thus, for all $i\in I$. Therefore, $\{(d_i(\cdot),n_i,0),i\in I\}$ satisfies Assumption \ref{ass:setB}. In order to show that the data $\{(d_i(\cdot),n_i,c_i)i\in\allN\}$ satisfies Assumption \ref{ass:projection}, we first need some definitions. Define the set-valued function $I(\cdot)$ on $G_{\bar x}$ by 
	\be\label{eq:Iy}I(y)\doteq\{i\in I:\ip{y,n_i}=0\}.\qquad y\in G_{\bar x},\ee 
and define the set-valued function $d_I(\cdot,\cdot)$ on $U\times G_{\bar x}$ by 
	\be\label{eq:dI}d_I(\alpha,y)\doteq\conv\lsb\{d_i(\alpha),i\in I(y)\}\rsb,\qquad \alpha\in U,\;y\in G_{\bar x}.\ee
In other words, $I(\cdot)$ and $d_I(\cdot,\cdot)$ are defined analogously to $\allN(\cdot)$ and $d(\cdot,\cdot)$, respectively, but with the data $\{(d_i(\cdot),n_i,0),i\in I\}$ instead of $\{(d_i(\cdot),n_i,c_i),i\in\allN\}$. Now, according to Assumption \ref{ass:projection} and Remark \ref{rmk:picontinuous}, for each $\alpha\in U$ there is a continuous map $\pi^\alpha:\R^J\mapsto G$ such that $\pi^\alpha(x)=x$ for all $x\in G$ and $\pi^\alpha(x)-x\in d(\alpha,\pi^\alpha(x))$ for all $x\not\in G$. By the upper semicontinuity of $\allN(\cdot)$ (Lemma \ref{lem:allNusc}), the continuity of $\pi^\alpha$ and because $\pi^\alpha(\bar x)=\bar x$, there exists a neighborhood $V_{\bar x}$ of $\bar x$ such that $\allN(\pi(x))\subseteq\allN(\pi(\bar x))=I$ for all $x\in V_{\bar x}$.  We now define a map $\pi_I^\alpha:\R^J\mapsto G_{\bar x}$. For $y\in G_{\bar x}$, set $\pi_I^\alpha(y)=y$. For $y\not\in G_{\bar x}$, choose $\delta>0$ such that $\bar x+\delta y\in V_{\bar x}$ and define 
	\be\label{eq:piI}\pi_I^\alpha(y)\doteq\delta^{-1}(\pi^\alpha(\bar x+\delta y)-\bar x).\ee
Since $\pi^\alpha(\bar x+\delta y)\in V_{\bar x}$, $\bar x\in F_I$, and by \eqref{eq:Iy}, we have $\pi^\alpha(\bar x+\delta y)-\bar x\in G_{\bar x}$ and
\begin{align*}
	\allN(\pi^\alpha(\bar x+\delta y))&=\{i\in I:\ip{\pi^\alpha(\bar x+\delta y),n_i}=c_i\}\\
	&=\{i\in I:\ip{\pi^\alpha(\bar x+\delta y)-\bar x,n_i}=0\}\\
	&=I(\pi^\alpha(\bar x+\delta y)-\bar x).
\end{align*}
Since $G_{\bar x}$ is a cone with vertex at the origin, it follows that $\delta^{-1}(\pi^\alpha(\bar x+\delta y)-\bar x)\in G_{\bar x}$ and 
	\be\label{eq:Idelta}\allN(\pi^\alpha(\bar x+\delta y))=I(\delta^{-1}(\pi^\alpha(\bar x+\delta y)-\bar x)).\ee
Thus, by \eqref{eq:dI}, \eqref{eq:Idelta} and \eqref{eq:dx},
	\be\label{eq:dIalphadelta}d_I(\alpha,\delta^{-1}(\pi^\alpha(\bar x+\delta y)-\bar x))=\conv\lsb\{d_i(\alpha),i\in\allN(\pi^\alpha(\bar x+\delta y))\}\rsb=d(\alpha,\pi^\alpha(\bar x+\delta y)).\ee
Then, by \eqref{eq:piI}, the facts that $\{(d_i(\cdot),n_i,c_i),i\in\allN\}$ satisfies Assumption \ref{ass:projection} and $d(\alpha,x)$ is a cone for all $x\in G$, and \eqref{eq:dIalphadelta}, we have
	$$\pi_I^\alpha(y)-y=\delta^{-1}(\pi^\alpha(\bar x+\delta y)-(\bar x+\delta y))\in d(\alpha,\pi^\alpha(\bar x+\delta y))=d_I(\alpha,\pi^\alpha(\bar x+\delta y)-\bar x).$$ 
Since this holds for all $\alpha\in U$ and $y\in G_{\bar x}$, $\{(d_i(\cdot),n_i,0),i\in I\}$ satisfies Assumption \ref{ass:projection}. The existence of a unique solution $(\z,\y)$ to the ESP $\{(d_i(\alpha),n_i,0),i\in I\}$ for $\x\in\cts_{G_{\bar x}}(\R^J)$ then follows from Proposition \ref{prop:esp}. The bound \eqref{eq:projlip} follows from \cite[Lemma 9.8]{Lipshutz2016}.
\end{proof}

\bibliographystyle{plain}
\bibliography{pathwise}

\end{document}